\documentclass[11pt]{article}
\usepackage[a4paper, left=1.5cm, right=1.5cm, top=2.5cm, bottom=2.5cm, headsep=1.2cm]{geometry}               
\usepackage[parfill]{parskip}   
\usepackage{graphicx}
\usepackage{amssymb}
\usepackage{color}
\usepackage{amsfonts}
\usepackage{amsmath} 
\usepackage{amsthm}
\usepackage{empheq}
\usepackage{esint}
\usepackage{epstopdf}
\usepackage{verbatim}
\usepackage{ifthen}
\usepackage{comment}
\usepackage{hyperref}

\begin{document}

\newtheorem{defi}{Definition}
\newtheorem{theo}{Theorem}
\newtheorem{lem}{Lemma}
\newtheorem{rem}{Remark}
\newcommand{\ov}[1]{\overline{#1}}
\newcommand{\un}[1]{\underline{#1}}
\newcommand{\bsym}[1]{\boldsymbol{#1}}
\newcommand{\n}[1]{\lVert#1\rVert}
\newcommand{\eq}[1]{\begin{align}#1\end{align}}

\providecommand{\bs}{\begin{subequations}}
\providecommand{\es}{\end{subequations}}

\title{A model of morphogen transport\\ in the presence of glypicans II}
\author{Marcin Ma\l ogrosz
\footnote{Institute of Applied Mathematics and Mechanics, University of Warsaw, Banacha 2, 02-097 Warsaw, Poland (malogrosz@mimuw.edu.pl)}}
\date{}

  \maketitle

\begin{abstract}
\noindent
The model of morphogen transport introduced by Hufnagel et al. consisting of two evolutionary PDEs of reaction-diffusion type and three ODEs posed on a rectangular domain $(-L,L)\times (0,\epsilon H)$ is analysed. We prove that the problem is globally well-posed and that the corresponding solutions converge as $\epsilon$ tends to $0$ to the unique solution of the one dimensional (i.e. posed on $(-L,L)$) system, which was analysed in previous paper. Main difficulties in the analysis stem from the presence of a singular source term - a Dirac Delta, combined with no smoothing effect in the ODE part of the system.  
\end{abstract}



%\begin{abstract}
%\noindent
%\end{abstract}

%\textbf{AMS classification} 35B40, 35Q92

%\textbf{Keywords} morphogen transport, reaction-diffusion equations, singular perturbation, singular boundary condition, semigroup theory,

\section{Introduction}

According to the French Flag Model, created in the late sixties by Lewis Wolpert (see \cite{Wol}), morphogen are molecules which due to mechanism of positional signalling govern the fate of cells in living organisms. It has been observed that certain proteins (and other substances) after being secreted from a source, typically a group of cells, spread through the tissue and after a certain amount of time form a stable gradient of concentration. Next receptors located on the surfaces of the cells detect levels of morfogen concentration and transmit these information to the nucleus. This leads to the activation of appropriate genes, synthesis of proteins and finally differentiation of cells. 

Although the role of morphogen in cell differentiation, as described above, is commonly accepted there is still discussion regarding the exact kinetic mechanism of the movement of morphogen molecules and the role of reactions of morphogen with receptors in forming the gradient of concentration (see \cite{GB},\cite{KPBKBJG-G},\cite{KW}).  To determine the mechanism of morphogen transport, several mathematical models consisting of systems of semilinear parabolic PDEs of reaction diffusion-type coupled with ODEs were recently proposed and analysed (see \cite{KLW1},\cite{KLW2},\cite{Tel},\cite{Mal1},\cite{STW}). 

In this series of papers we analyse model \textbf{[HKCS]}, introduced by Hufnagel et al. in\cite{Huf}, which describes the formation of the gradient of morphogen Wingless (Wg) in the imaginal wing disc of the Drosophila Melanogaster individual. Model \textbf{[HKCS]} has two counterparts - one and two dimensional, depending on the dimensionality of the domain representing the imaginal wing disc. We denote these models \textbf{[HKCS].1D} and \textbf{[HKCS].2D} respectively. In mathematical terms \textbf{[HKCS].1D} is a system of two semilinear parabolic PDEs of reaction diffusion type coupled with three nonlinear ODEs posed on the interval $I^L=(-L,L)$, while \textbf{[HKCS].2D} consists of a linear parabolic PDE posed on rectangle $\Omega^{L,H}=(-L,L)\times(0,H)$ which is coupled via nonlinear boundary condition on $\partial_1{\Omega^{L,H}}=(-L,L)\times\{0\}$ with a semilinear parabolic PDE and three ODEs. 

In \cite{Mal2} we have shown that \textbf{[HKCS].1D} is globally well posed and has a unique stationary solution. In this paper we turn our attention to the analysis of the \textbf{[HKCS].2D} model. Using analytic semigroup theory we prove its global well-posedness in appropriately chosen function setting and justify rigourously that \textbf{[HKCS].1D} can be obtained from \textbf{[HKCS].2D} through "ironing of the wing disc" - i.e. dimension reduction of the domain in the direction perpendicular to the surface of the wing disc. The main analytic problem which we have to overcome stems from two factors: the lack of smoothing effect in the ODEs and the presence of a point source term (a Dirac Delta) in the boundary condition for the equation posed on $(-L,L)\times(0,H)$, which causes the solution to be unbounded for every $t>0$. To overcome this difficulty a new notion of solution is proposed in  Section \ref{H2E:secM-mild}. Stationary problem for the \textbf{[HKCS].2D} is analysed in \cite{Mal4}.  

\subsection{The [HKCS].2D model.}

In this section we present the model \textbf{[HKCS].2D} - a two dimensional counterpart of the \textbf{[HKCS]} model introduced in $\cite{Huf}$. For the presentation and analysis of \textbf{[HKCS].1D} - a one dimensional counterpart we refer to \cite{Mal2}. 

For $L,H>0, \ \infty\geq T>0$, denote 
\begin{align*}
I^L&=(-L,L), \ I^1=(-1,1),\\ 
\Omega^{L,H}&=(-L,L)\times(0,H), \ \partial \Omega^{L,H}=\partial_1\Omega^{L,H}\cup\partial_0\Omega^{L,H}, \ \partial_1\Omega^{L,H}=(-L,L)\times\{0\}, \\ \Omega^{L,H}_T&=(0,T)\times \Omega^{L,H}, \ (\partial \Omega^{L,H})_T=(0,T)\times\partial \Omega^{L,H}, \ \Omega=\Omega^{1,1}.
\end{align*}
The domain $\Omega^{L,H}$ represents the imaginal wing disc of the Drosophila Melanogaster individual and the $x_2$ direction corresponds to the thickness of the disc, so that in practice $H\ll L$.
Let $\nu$ denote a unit outer normal vector to $\partial\Omega^{L,H}$ and let $\delta$ be a one dimensional Dirac Delta located at $x=0$ (that is $\delta(\phi)=\phi(0)$ for any $\phi\in C([-L,L])$).

\begin{figure}[h]
\begin{center}
\includegraphics[scale=2, trim=0 75 0 20]{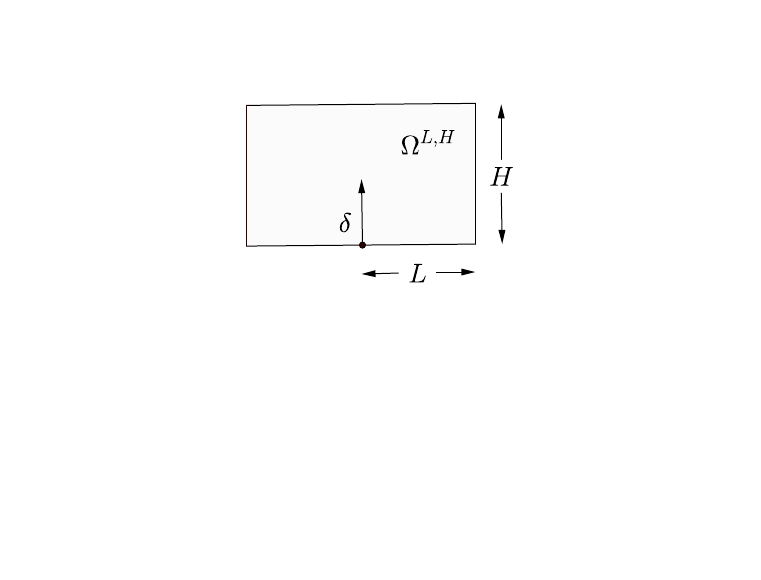}
\end{center}
\caption{Graph of the domain $\Omega^{L,H}$. The arrow pointing towards the rectangle represents a point source of the morphogen (a Dirac Delta) on the boundary.} 
\label{H2E:wykres}
\end{figure}
\textbf{[HKCS].2D}\\
 is a system which consists of one evolutionary PDE posed on $\Omega^{L,H}$, one evolutionary PDE and 3 ODE's posed on $\partial_1\Omega^{L,H}$:
\bs\label{H2E:Model}
\eq{
\partial_tW-D\Delta W&=-\gamma W, &&(t,x)\in \Omega^{L,H}_{T}\\
\partial_tW^*-D^*\partial^2_{x_1}W^*&=-\gamma^* W^*+[kGW-k'W^*]-[k_{Rg}RW^*-k_{Rg}'R_g^*], &&(t,x)\in (\partial_1\Omega^{L,H})_{T}\\
\partial_tR&= -[k_RRW-k_R'R^*]-[k_{Rg}RW^*-k_{Rg}'R_g^*]-\alpha R+\Gamma, &&(t,x)\in (\partial_1\Omega^{L,H})_{T}\\
\partial_tR^*&=[k_RRW-k_R'R^*] -\alpha^*R^*, &&(t,x)\in (\partial_1\Omega^{L,H})_{T}\\
\partial_tR_{g}^*&=[k_{Rg}RW^*-k_{Rg}'R_g^*]-\alpha^*R_{g}^*, &&(t,x)\in (\partial_1\Omega^{L,H})_{T}
}
\es

supplemented by the boundary conditions:
\bs\label{H2E:ModelBCD}
\eq{
D\nabla W\nu&=0, &&(t,x)\in(\partial_0\Omega^{L,H})_{T}\label{H2E:ModelBCD1}\\
D\nabla W\nu&=- [kGW-k'W^*] - [k_RRW-k_R'R^*] + s\delta,&&(t,x)\in(\partial_1\Omega^{L,H})_{T}\label{H2E:ModelBCD2}\\
\partial_{x_1}W^*&=0, &&(t,x)\in(\partial\partial_1\Omega^{L,H})_{T}\label{H2E:ModelBCD3}
}
\es
and initial conditions:
\bs\label{H2E:ModelINI}
\eq{
W(0)&=W_0, && x\in \Omega^{L,H}\\ 
W^*(0)&=W^*_0,\  R(0)=R_0,\ R^*(0)=R^*_0,\ R_g^*(0)=R_{g0}^*,  && x\in \partial_1\Omega^{L,H}
}
\es

In \eqref{H2E:Model},\eqref{H2E:ModelBCD},\eqref{H2E:ModelINI} $W, G, R$ denote concentrations of free morphogens Wg, 
free glypicans Dlp and free receptors,\\ $W^*, R^*$ denote concentrations of morphogen-glypican
and morphogen-receptor complexes,\\ $R_g^*$ denotes concentration of morphogen-glypican-receptor complexes. It is assumed that $W$ is located on $\Omega^{L,H}$, while other substances are present only on $\partial_1\Omega^{L,H}$.
The model takes into account association-dissociation mechanism of
\begin{itemize}
\item $W$ and $G$ with rates $k, k': \ kGW-k'W^*$.
\item $W$ and $R$ with rates $k_R, k_R': \ k_RRW-k_R'R^*$.
\item $W^*$ and $R$ with rates $k_{Rg}, k_{Rg}': \ k_{Rg}RW^*-k_{Rg}'R_g^*$.
\end{itemize}

Other terms of the system account for
\begin{itemize}
\item diffusion of $W$ in $\Omega^{L,H}$ (resp. $W^*$ on $\partial_1\Omega^{L,H}$) with rate $D$ (resp $D^*$): $-D\Delta W$ (resp. $-D^*\partial^2_{x_1}W^*$).
\item Degradation of $W$ in $\Omega^{L,H}$ (resp. $W^*$ on $\partial_1\Omega^{L,H}$) with rate $\gamma$ (resp. $\gamma^*$): $-\gamma W$ (resp. $-\gamma^* W^*$).
\item Endocytosis of $R$ (resp. $R^*,R_g^*$) with rate $\alpha$ (resp. $\alpha^*,\alpha^*$) : $-\alpha R$ (resp. $-\alpha^* R^*, -\alpha^* R_g^*$).
\item Secretion of $W$ with rate $s$ from the point source localised at $x=0\in\partial_1\Omega^{L,H}$: $s\delta$.
\item Production of $R$: \ $\Gamma$.
\end{itemize}

For simplicity we assume that $G$ and $\Gamma$ are given and constant (in time and space).

In order to analyse the reduction of the dimension of the domain we introduce for $\epsilon>0$ the \textbf{[HKCS].(2D,$\epsilon$)} model, which is obtained from \textbf{[HKCS].2D} by changing $\Omega^{L,H}$ into $\Omega^{L,\epsilon H}$ and rescaling the source in the boundary conditions \eqref{H2E:ModelBCD}:   

\textbf{[HKCS].(2D,$\epsilon$)}
\bs\label{H2E:Modele}
\eq{
\partial_tW^{\epsilon}-D\Delta W^{\epsilon}&=-\gamma W^{\epsilon}, &&(t,x)\in \Omega^{L,\epsilon H}_{T}\nonumber\\
\partial_tW^{*,\epsilon}-D^*\partial^2_{x_1}W^{*,\epsilon}&=-\gamma^* W^{*,\epsilon}+[kGW^{\epsilon}-k'W^{*,\epsilon}]-[k_{Rg}R^{\epsilon}W^{*,\epsilon}-k_{Rg}'R_g^{*,\epsilon}], &&(t,x)\in (\partial_1\Omega^{L,\epsilon H})_{T}\nonumber\\
\partial_tR^{\epsilon}&= -[k_RR^{\epsilon}W^{\epsilon}-k_R'R^{*,\epsilon}]-[k_{Rg}R^{\epsilon}W^{*,\epsilon}-k_{Rg}'R_g^{*,\epsilon}]-\alpha R^{\epsilon}+\Gamma, &&(t,x)\in (\partial_1\Omega^{L,\epsilon H})_{T}\nonumber\\
\partial_tR^{*,\epsilon}&=[k_RR^{\epsilon}W^{\epsilon}-k_R'R^{*,\epsilon}] -\alpha^*R^{*,\epsilon}, &&(t,x)\in (\partial_1\Omega^{L,\epsilon H})_{T}\nonumber\\
\partial_tR_{g}^{*,\epsilon}&=[k_{Rg}R^{\epsilon}W^{*,\epsilon}-k_{Rg}'R_g^{*,\epsilon}]-\alpha^*R_{g}^{*,\epsilon}, &&(t,x)\in (\partial_1\Omega^{L,\epsilon H})_{T}\nonumber
}
\es

with boundary conditions
\bs\label{H2E:ModeleBCD}
\eq{
\epsilon^{-1}D\nabla W^{\epsilon}\nu&=0, &&(t,x)\in(\partial_0\Omega^{L,\epsilon H})_{T}\nonumber\\
\epsilon^{-1}D\nabla W^{\epsilon}\nu&=- [kGW^{\epsilon}-k'W^{*,\epsilon}] - [k_RR^{\epsilon}W^{\epsilon}-k_R'R^{*,\epsilon}] + s\delta,&&(t,x)\in(\partial_1\Omega^{L,\epsilon H})_{T}\nonumber\\
\partial_{x_1}W^{*,\epsilon}&=0, &&(t,x)\in(\partial\partial_1\Omega^{L,\epsilon H})_{T}\nonumber
}
\es
and initial conditions
\begin{align*}
W^{\epsilon}(0)&=W_0^{\epsilon}, && x\in \Omega^{L,\epsilon H}\\ 
W^{*,\epsilon}(0)&=W^*_0,\ R^{\epsilon}(0)=R_0, \ R^{*,\epsilon}(0)=R^*_0,\ R_g^{*,\epsilon}(0)=R_{g0}^*, && x\in \partial_1\Omega^{L,\epsilon H},
\end{align*}
where $W_0^{\epsilon}(x_1,x_2)=W_0(x_1,x_2/\epsilon)$.\\

Observe that  \textbf{[HKCS].(2D,1)}=\textbf{[HKCS].2D}. Roughly speaking besides the well-posedness of \textbf{[HKCS].2D}, our main result is that 
\begin{align}
\lim_{\epsilon\to0^+}\textbf{[HKCS].(2D,$\epsilon$)}=\textbf{[HKCS].1D}\label{H2E:modellimit},
\end{align}
where \textbf{[HKCS].1D} was analysed in \cite{Mal2}. The precise meaning of the limit \eqref{H2E:modellimit} is given in Theorem \ref{H2E:Maintheorem2}.

\subsection{Nondimensionalisation}
Introduce the following nondimensional parameters:
\begin{align*}
&T=L^2/D, \ K_1=k_RT, \ K_2=k_RT/H, \ h=\epsilon H/L, \ d=D^*/D,\\
&\bsym{b}=(b_1,b_2,b_3,b_4,b_5)=(T\gamma,T\gamma^*,T\alpha,T\alpha^*,T\alpha^*),\\
&\bsym{c}=(c_1,c_2,c_3,c_4,c_5)=(TkG/H,Tk',Hk_{Rg}/k_R,Tk_R',Tk_{Rg}'),\\
&\bsym{p}=(p_1,p_2,p_3,p_4,p_5)=(K_2Ts,0,K_2T\Gamma,0,0).
\end{align*}
For $(t,x)=(t,x_1,x_2)\in\Omega_T=(0,T)\times(-1,1)\times(0,1)$ we define functions
\begin{align*}
&u_{1}^h(t,x_1,x_2)=K_1W^{\epsilon}(Tt,Lx_1,\epsilon Hx_2), \ u_{2}^h(t,x_1)=K_2W^{*,\epsilon}(Tt,Lx_1), \ u_{3}^h(t,x_1)=K_2R^{\epsilon}(Tt,Lx_1),\\
&u_{4}^h(t,x_1)=K_2R^{*,\epsilon}(Tt,Lx_1),\ u_{5}^h(t,x_1)=K_2R_g^{*,\epsilon}(Tt,Lx_1),\\
&\bsym{u}^h=(u_{1}^h,u_{2}^h,u_{3}^h,u_{4}^h,u_{5}^h),\\
&u_{01}(x_1,x_2)=K_1W_0^{\epsilon}(Lx_1,\epsilon Hx_2)=K_1W_0(Lx_1,Hx_2),\ u_{02}(x_1)=K_2W^*_0(Lx_1), \ u_{03}(x_1)=K_2R_0(Lx_1),\\
&u_{04}(x_1)=K_2R^*_0(Lx_1), \ u_{05}(x_1)=K_2R_{g0}^*(Lx_1),\\
&\bsym{u}_0=(u_{01},u_{02},u_{03},u_{04},u_{05}),
\end{align*}
then system \textbf{[HKCS].(2D,$\epsilon$)} rewritten in the nondimensional form reads
\bs\label{H2E:System}
\eq{
\partial_t u_{1}^h+div(J_h(u_{1}^h))&=-b_1u_{1}^h,&& (t,x)\in\Omega_{T}\label{H2E:SystemA}\\
\partial_t u_{2}^h-d\partial^2_{x_1} u_{2}^h&=c_1u_{1}^h-(b_2+c_2+c_3u_{3}^h)u_{2}^h+c_5u_{5}^h,&&(t,x)\in(\partial_1\Omega)_{T}\label{H2E:SystemB}\\
\partial_t u_{3}^h&=-(b_3+u_{1}^h+c_3u_{2}^h)u_{3}^h+c_4u_{4}^h+c_5u_{5}^h+p_3,&&(t,x)\in(\partial_1\Omega)_{T}\label{H2E:SystemC}\\
\partial_t u_{4}^h&=u_{1}^hu_{3}^h-(b_4+c_4)u_{4}^h,&&(t,x)\in(\partial_1\Omega)_{T}\label{H2E:SystemD}\\
\partial_t u_{5}^h&=c_3u_{2}^hu_{3}^h-(b_5+c_5)u_{5}^h,&&(t,x)\in(\partial_1\Omega)_{T}\label{H2E:SystemE}
}
\es
with boundary and initial conditions
\begin{align*}
-J_h(u_{1}^h)\nu&=0,&&(t,x)\in(\partial_0\Omega)_{T}\\
-J_h(u_{1}^h)\nu&=-(c_1+u_{3}^h)u_{1}^h+c_2u_{2}^h+c_4u_{4}^h+p_1\delta,&&(t,x)\in(\partial_1\Omega)_{T}\\
\partial_{x_1} u_{2}^h&=0,&&(t,x)\in(\partial\partial_1\Omega)_{T}\\
\boldsymbol{u}^h(0,\cdot)&=\boldsymbol{u}_0,
\end{align*}
where 
\begin{itemize}
\item $J_h(u)=-(\partial_{x_1}u,h^{-2}\partial_{x_2}u)$ denotes the flux of $u_1^h$,
\item  $\nu$ denotes the outer normal unit vector to $\partial\Omega$,
\item $\delta$ denotes a one dimensional Dirac Delta i.e $\delta(\phi)=\phi(0)$ for any $\phi\in C([-1,1])$.
\end{itemize}

From now on we impose the following natural assumptions on the signs of the constant parameters and (possibly nonconstant) initial condition
\begin{align}
d, \boldsymbol{b}>0,\  \boldsymbol{c},\boldsymbol{p}, \boldsymbol{u_0}\geq 0.\label{H2E:ass}
\end{align}

\subsection{Overview}

The paper is divided into five sections as follows

\tableofcontents

\begin{comment}
\begin{enumerate}
\item Introduction
\begin{enumerate}
\item[1.1] The \textbf{[HKCS].2D} model
\item[1.2] Nondimensionalisation
\item[1.3] Overwiev
\end{enumerate}
\item Notation
\item Analytical Tools
\begin{enumerate}
\item[3.1] Inequalities
\item[3.2] Existence result for system of abstract ODEs
\item[3.3] Operators, semigroups, estimates
\item[3.4] Auxiliary functions
\end{enumerate}
\item The case of regular source
\item The case of singular source
\begin{enumerate}
\item[5.1] Definition of M-mild solution
\item[5.2] The main results - well-posedness (Theorem \ref{H2E:Maintheorem1}) and dimension reduction (Theorem  \ref{H2E:Maintheorem2})
\item[5.3] Proof of Theorem \ref{H2E:Maintheorem1}   
\item[5.4] Proof of Theorem \ref{H2E:Maintheorem2}
\end{enumerate}
\end{enumerate}
\end{comment}

In Section \ref{H2E:secAnalytical} we introduce analytical tools which are used in the following sections. In Section \ref{H2E:secInequalities} we collect and prove several inequalities including a Gronwall type inequality which to our best knowlege didn't appear in the literature previously.  In Section \ref{H2E:secex} we recall the definition of $\mathcal{X}^{\alpha}$ solution and state the basic result concerning existence of $\mathcal{X}^{\alpha}$ solutions to systems of abstract ODEs in Banach spaces.  In Section \ref{H2E:secOperators} we develop theory of the $L_2$ realisation of the operator $-div(J_h(u))=[\partial^2_{x_1x_1}+h^{-2}\partial^2_{x_2x_2}]u$ with Neumann boundary condition on the rectangle $\Omega$. The main result of Section \ref{H2E:secOperators} is Lemma \ref{H2E:lemiron} concerning semigroup estimates used in the dimension reduction of System \eqref{H2E:System} (see Section \ref{H2E:secdimred}).

In section \ref{H2E:secregular} we study the well-posedness of the aproximation of System \ref{H2E:System} which is obtained by regularisation of the singular source term $\delta$. The main result of this section is Theorem \ref{H2E:existencereg} in which we prove that the regularised problem has a unique globally in time defined $\mathcal{X}^{\alpha}$ solution.

In Section \ref{H2E:secdimred} we study the well-posedness of system \eqref{H2E:System}. Moreover, we perform the dimension reduction which is obtained by the limit passage $h\to0^+$. In section \ref{H2E:secAuxiliary} we introduce auxilliary functions which play fundamental role in the definition of M-mild solution given in section \ref{H2E:secM-mild}. Next in section \ref{H2E:secMain} we state the main results of the paper: in Theorem \ref{H2E:Maintheorem1} we prove that system \eqref{H2E:System} has a unique global M-mild solution, while in Theorem \ref{H2E:Maintheorem2} we show that this solution converges as $h\to0^+$ to the solution of an appropriate system posed on a one dimensional interval. 

\section{Notation}\label{H2E:Notation}

If $X$ is a linear space and $Y$ is an arbitrary subset of $X$ we denote by $lin(Y)$ the linear space which consists of all linear combinations of elements of $Y$. If $X$ is a topological space and $U$ is an arbitrary subset of $X$ we denote by $cl_{X}(U)$ the closure of $U$ in $X$. If $X$ is a normed vector space we denote by $\n{\cdot}_X$ its norm and by $X^*$ its topological dual. If $x\in X$ and $x^*\in X^*$ we denote by $\Big<x^*,x\Big>_{(X^*,X)}=x^*(x)$ a natural pairing between $X$ and $X^*$. If $H$ is a Hilbert space we denote by $(.|.)_H$ its scalar product. In particular $(\bsym{x}|\bsym{y})_{\mathbb{R}^n}=\sum_{i=1}^nx_iy_i$ and $(f|g)_{L_2(U)}=\int_{U}fg$. To get more familiar with the notation observe that, due to Riesz theorem, for any $x^*\in H^*$ there exists a unique $x\in H$ such that $\Big<x^*,y\Big>_{(H^*,H)}=(x|y)_H$ for all $y\in H$. If $H_1$ and $H_2$ are Hilbert spaces and $u\in H_1, \ v\in H_2$ we denote by $H_1\otimes H_2$ and $H\otimes H$ tensor products. For a comprehensive treatment on normed, Hilbert and Banach spaces we refer to [\cite{Reed}, Chap. I-III].

If $X,Y$ are normed spaces we denote by $\mathcal{L}(X,Y)$ the Banach space of bounded linear operators between $X$ and $Y$, topologised by the operator norm. We write $X\subset Y$ to denote a continuous imbedding of $X$ into $Y$, while the compact imbedding is denoted $X\subset\subset Y$. If $A\in\mathcal{L}(X,Y)$ then we denote by $A'\in\mathcal{L}(Y^*,X^*)$ the transpose of $A$. If $D(A)$ is a linear (not necessarily dense) subspace of a Banach space $X$  we write $A:X\supset D(A)\to X$ when $A$ is an unbounded linear operator with domain $D(A)$. We denote by $G(A)$ the graph of $A$ (i.e. the set $\{(u,Au): u\in D(A)\}$). 
For a closed, unbounded operator $A: X\supset D(A)\to X$ we denote by $\rho(A)$ the resolvent set of $A$, by $\sigma(A)=\mathbb{C}\setminus \rho(A)$ its spectrum and by $\sigma_p(A)$ its pure point spectrum (i.e. set of eigenvalues). The resolvent operator is denoted by $R(\lambda,A)$ or $(\lambda-A)^{-1}$ for $\lambda\in\rho(A)$. If $H_1,H_2$ are Hilbert spaces and $A:H_1\supset D(A)\to H_2$ is a densely defined unbounded operator we denote by $A^*:H_2\supset D(A^*)\to H_1$ the adjoint of $A$. Let us point out that if $X$ is a real vector space than by the spectrum of operator $A$ we understand the spectrum of its complexification. For a comprehensive treatment on bounded and unbounded linear operators as well as their spectral theory we refer to [\cite{Reed}, Chap. VI-VIII]. 

If $U$ is a subset of $\mathbb{R}^n$ we denote by $\overline{U}=cl_{\mathbb{R}^n}(U)$ its closure and by $\partial U$ its boundary. For $1\leq p\leq \infty, \ s\in\mathbb{R}$, we denote by $W^{s}_p(U)$ the fractional Sobolev (also known as Sobolev-Slobodecki) spaces. If $U$ is also bounded we denote by $C(\overline{U})$ the Banach space of continuous functions on $\overline{U}$ topologised by the supremum norm which we denote by $\n{\cdot}_{\infty}$. We will also use $\n{\cdot}_{\infty}$ to denote the essential supremum norm used in the definition of space $L_{\infty}(U)$. By $\mathcal{M}(\overline{U})=(C(\overline{U}))^*$ we denote the Banach space of finite, signed Radon measures. For $\theta\in[0,1]$ we denote by  $[\cdot,\cdot]_{\theta}$ the complex interpolation functor. For a comprehensive treatment on spaces $W^s_p(U)$ and functor $[\cdot,\cdot]_{\theta}$ we refer to \cite{Tri}.

In the whole article $I_+=(0,1), \ I=(-1,1)$ and $\Omega=(-1,1)\times(0,1)$ are fixed domains. Moreover we denote by $\delta$ a one dimensional Dirac Delta: $ \ \delta\in\mathcal{M}(\overline{I})$ and $\delta(\phi)=\int_{\overline{I}}\phi d\delta=\phi(0)$ for any $\phi\in C(\overline{I})$. For $i,j\in\mathbb{N}$ we denote by $\delta_{ij}$ the Kronecker symbol i.e. $\delta_{ij}=\begin{cases}1 &  {\rm if} \ i=j\\ 0 & {\rm if} \ i\neq j\end{cases}$.
 
In estimates we will use a generic constant $C$ which may take different values even in the same paragraph. Constant $C$ may depend on various parameters, but it will never depend on $h$  nor any other parameter which could change due to a limitting process.

\section{Analytical Tools}\label{H2E:secAnalytical}
\subsection{Inequalities}\label{H2E:secInequalities}

In Lemma \ref{H2E:ineq} we collect three elementary estimates which are used in the following chapters. For completeness of the reasoning we provide short proofs. \\

\begin{lem}\label{H2E:ineq}
The following inequalities hold
\bs\label{H2E:ineqs}
\eq{
\sup\{t^{\alpha}e^{-rt}: t\geq t_0\}&\leq C(r^{-\alpha}+t_0^{\alpha})e^{-rt_0}, &&t_0\geq0, \alpha\geq0, r>0, \label{H2E:ineq1}\\
\int_{0}^t\frac{d\tau}{\tau^{\alpha}(t-\tau)^{\beta}}&\leq Ct^{1-\alpha-\beta}, && t>0, \alpha,\beta\geq 0, \alpha+\beta<1,\label{H2E:ineq2}\\
\int_{0}^te^{-r\tau}\frac{d\tau}{\tau^{\alpha}(t-\tau)^{\beta}}&\leq C\Big(\frac{t^{\alpha+\beta}}{r}\Big)^{\frac{1-(\alpha+\beta)}{1+\alpha+\beta}}, && t>0, \alpha,\beta\geq 0, \alpha+\beta<1, r>0, \label{H2E:ineq3}
}
where constant $C$ depends only on $\alpha$ and $\beta$.
\es
\end{lem}

\begin{proof}
To prove \eqref{H2E:ineq1} define for $t\geq0$ function $f(t)=t^{\alpha}e^{-rt}$. Then $f'(t)=\alpha t^{\alpha-1}e^{-rt}-rt^{\alpha}e^{-rt}=t^{\alpha-1}e^{-rt}(\alpha-rt)$. Analysing the sign of $f'$ we obtain that function $f$ is increasing on $[0,\alpha/r]$ and decreasing on $[\alpha/r,\infty)$. It follows that 
\begin{align*}
\sup\{t^{\alpha}e^{-rt}: t\geq t_0\}=\left\{ \begin{array}{ll}
f(\alpha/r) & \textrm{if $t_0\leq\alpha/r$}\\
f(t_0) & \textrm{if $t_0>\alpha/r$}\\
\end{array} \right.\leq f(\alpha/r)+f(t_0)\leq C(r^{-\alpha}+t_0^{\alpha})e^{-rt_0},
\end{align*}
where one can take $C=\max\{\alpha^\alpha,1\}$. 
To prove inequality \eqref{H2E:ineq2} we change variables $\tau=ty$. Then we have
\begin{align*}
\int_0^t\frac{d\tau}{\tau^{\alpha}(t-\tau)^{\beta}}=t^{1-\alpha-\beta}\int_{0}^1\frac{dy}{y^{\alpha}(1-y)^{\beta}}\leq Ct^{1-\alpha-\beta}.
\end{align*}
Finally we prove \eqref{H2E:ineq3}. Set $q=\frac{\alpha+\beta+1}{2(\alpha+\beta)}$. It is easy to check that $1<q$ and $(\alpha+\beta)q<1$. Let $p=\frac{q}{q-1}=\frac{1+(\alpha+\beta)}{1-(\alpha+\beta)}$ be $q$'s H\"older conjugate exponent. Using H\"older inequality we obtain
\begin{align*}
&\int_{0}^te^{-r\tau}\frac{1}{\tau^{\alpha}(t-\tau)^{\beta}}d\tau=t^{1-(\alpha+\beta)}\int_0^1e^{-rty}\frac{1}{y^{\alpha}(1-y)^{\beta}}dy\leq t^{1-(\alpha+\beta)}\Big(\int_0^1e^{-prty}dy\Big)^{1/p}\Big(\int_0^1\frac{dy}{y^{q\alpha}(1-y)^{q\beta}}\Big)^{1/q}\\
&\leq Ct^{1-(\alpha+\beta)}\Big(\frac{1-e^{-rtp}}{rtp}\Big)^{1/p}\leq C\frac{t^{1-(\alpha+\beta)-1/p}}{r^{1/p}}=C\frac{t^{(\alpha+\beta)\frac{1-(\alpha+\beta)}{1+\alpha+\beta}}}{r^{\frac{1-(\alpha+\beta)}{1+\alpha+\beta}}}=C\Big(\frac{t^{\alpha+\beta}}{r}\Big)^{\frac{1-(\alpha+\beta)}{1+\alpha+\beta}}.
\end{align*}

\begin{comment}
$\frac{e^x-1}{x}=\sum_{k=0}^{\infty}x^{k}/(k+1)!\leq e^x\\
\frac{1-e^{-x}}{x}=e^{-x}\frac{e^x-1}{x}\leq\min\{1,1/x\}\\
t^{1-(\alpha+\beta)}\Big(\frac{1-e^{-rtp}}{rtp}\Big)^{1/p}\leq t^{1-(\alpha+\beta)}\min\{1,1/(rtp)^{1/p}\}$
\end{comment}
\end{proof}

Lemma \ref{H2E:Gronlem} is an extension of the well known Gronwall inequality in integral form. Although several results of similar type can be found in the  literature (for instance in \cite{Med}), we were not able to find a reference to the one which would cover the full range of parameters. Our method of proof is taken from \cite{Med}. \\  

\begin{lem}\label{H2E:Gronlem}
Let $0\leq\alpha,\beta, \ \alpha+\beta<1, \ 0\leq a, \ 0<b, \ 0<T<\infty$. Assume that $f\in L_{\infty}(0,T')$ for every $T'<T$ and that for a.e. $t\in(0,T)$ the following inequality holds
\begin{align*}
0\leq f(t)\leq a+b\int_0^t\frac{f(\tau)}{\tau^{\alpha}(t-\tau)^{\beta}}d\tau,
\end{align*}
then $f\in L_{\infty}(0,T)$ and
\begin{align*}
\n{f}_{L_{\infty}(0,T)}\leq Ca\exp\Big(Cb^{\frac{1+\alpha+\beta}{1-\alpha-\beta}}T^{1+\alpha+\beta}\Big),
\end{align*}
where $C$ depends only on $\alpha$ and $\beta$. Moreover $C=1$ when $\alpha=\beta=0$.
\end{lem}

\begin{proof}
When $\alpha=\beta=0$ the result is the well known Gronwall inequality in integral form. Otherwise we proceed similarly as in the proof of inequality \eqref{H2E:ineq3}.
Fix $q>1$ such that $q(\alpha+\beta)<1$ and let $p=\frac{q}{q-1}$ be $q$'s H\"older conjugate exponent. Using H\"older inequality we obtain
\begin{align*}
\int_0^t\frac{f(\tau)}{\tau^{\alpha}(t-\tau)^{\beta}}d\tau&\leq\Big(\int_0^tf(\tau)^pd\tau\Big)^{1/p}\Big(\int_0^t\frac{d\tau}{\tau^{\alpha q}(t-\tau)^{\beta q}}\Big)^{1/q}\\
&=\Big(\int_0^tf(\tau)^pd\tau\Big)^{1/p}t^{1/q-(\alpha+\beta)}\Big(\int_0^1\frac{d\tau}{\tau^{\alpha q}(1-\tau)^{\beta q}}\Big)^{1/q}\\
&\leq C_0T^{1/q-(\alpha+\beta)}\Big(\int_0^tf(\tau)^pd\tau\Big)^{1/p},
\end{align*}
where $C_0=\int_0^1\frac{d\tau}{\tau^{\alpha q}(1-\tau)^{\beta q}}$.
Thus
\begin{align*}
f(t)^p\leq \Big(a+bC_0T^{1/q-(\alpha+\beta)}\Big(\int_0^tf(\tau)^pd\tau\Big)^{1/p}\Big)^p\leq 2^{p-1}a^p+2^{p-1}b^pC_0^pT^{p/q-p(\alpha+\beta)}\int_0^tf(\tau)^pd\tau.
\end{align*}
Using Lemma \ref{H2E:Gronlem} with $\alpha=\beta=0$ we obtain
\begin{align*}
f(t)^p&\leq 2^{p-1}a^p\exp\Big(2^{p-1}b^pC_0^pT^{p/q-p(\alpha+\beta)}t\Big),\\
f(t)&\leq 2^{1/q}a\exp\Big(p^{-1}2^{p-1}b^pC_0^pT^{p/q-p(\alpha+\beta)+1}\Big)= 2^{1/q}a\exp\Big(p^{-1}2^{p-1}b^pC_0^pT^{p(1-\alpha-\beta)}\Big)\\
&\leq Ca\exp(Cb^pT^{p(1-\alpha-\beta)}),
\end{align*}
with $C=\max\{2^{1/q},p^{-1}2^{p-1}C_0^p\}$. To finish the proof observe that for $q=\frac{1}{2}\Big(1+\frac{1}{\alpha+\beta}\Big)$ one has $p=\frac{1+\alpha+\beta}{1-\alpha-\beta}$.
\end{proof}

\subsection{Existence result for a system of abstract ODE's}\label{H2E:secex}

For $i=1,\ldots, n$ let $(\mathcal{X}_i,\mathcal{X}_i^1)$ be a densely injected Banach couple (i.e. $\mathcal{X}_i^1$ is a dense subspace of $\mathcal{X}_i$ in the topology of $\mathcal{X}_i$). For $\alpha_i\in(0,1)$ denote $\mathcal{X}_i^{\alpha_i}=[\mathcal{X}_i,\mathcal{X}_i^1]_{\alpha_i}$ (where $[.,.]_{\alpha_i}$ is the complex interpolation functor). Finally note
\begin{align}
\alpha=(\alpha_1,\ldots,\alpha_n), \
\mathcal{X}=\mathcal{X}_1\times\ldots\times\mathcal{X}_n, \ 
\mathcal{X}^1=\mathcal{X}_1^1\times\ldots\times\mathcal{X}_n^1, \ 
\mathcal{X}^{\alpha}=\mathcal{X}_1^{\alpha_1}\times\ldots\times\mathcal{X}_n^{\alpha_n}.
\end{align}

\begin{lem}\label{H2E:exlem}

Assume that for $i=1,\ldots, n$ the following three conditions are satisfied
\begin{enumerate}
\item Operator $\mathcal{A}_i:\mathcal{X}_i\supset\mathcal{X}_i^1\to\mathcal{X}_i$ generates an analytic strongly continuous semigroup $e^{t\mathcal{A}_i}$. 
\item Map $\mathcal{F}_i: \mathcal{X}^{\alpha}\to \mathcal{X}_i$ is Lipschitz on bounded sets i.e.
\begin{align*}
\forall_{R>0}\exists_{C_R} \ \Big[\n{\bsym{u}}_{\mathcal{X}^{\alpha}},\n{\bsym{w}}_{\mathcal{X}^{\alpha}}\leq R\Longrightarrow\n{\mathcal{F}_i(\bsym{u})-\mathcal{F}_i(\bsym{w})}_{\mathcal{X}_i}\leq C_R\n{\bsym{u}-\bsym{w}}_{\mathcal{X}^{\alpha}}\Big]
\end{align*}
\item $u_{0i}\in\mathcal{X}_i^{\alpha_i}$.
\end{enumerate}
Then the following system of abstract ODE's
\begin{align}
\frac{d}{dt}u_i-\mathcal{A}_iu_i&=\mathcal{F}_i(\bsym{u}), \ t>0\label{H2E:abs1}\\
u_i(0)&=u_{0i}\label{H2E:abs2}
\end{align}
has a unique maximal $\mathcal{X}^{\alpha}$ solution $\bsym{u}=(u_1,\ldots,u_n)$ i.e. there exists a unique
\begin{align*}
\bsym{u}\in C([0,T_{\max});\mathcal{X}^{\alpha})\cap C^1((0,T_{\max});\mathcal{X})\cap C((0,T_{\max});\mathcal{X}^1),
\end{align*}
which satisfies system \eqref{H2E:abs1}-\eqref{H2E:abs2} in the classical sense.
For $t\in(0,T_{\max})$ the following Duhamel formulas hold:
\begin{align*}
u_{i}(t)=e^{t\mathcal{A}_i}u_{0i}+\int_0^te^{(t-\tau)\mathcal{A}_i}\mathcal{F}_i(\bsym{u}(\tau))d\tau, \ 1\leq i\leq n.
\end{align*}
and $T_{\max}$ satisfies the blow-up condition:
\begin{align}
{\rm if} \ T_{\max}<\infty \ {\rm then} \ \limsup_{t\to T_{\max}^{-}}\n{\bsym{u}(t)}_{\mathcal{X}^{\alpha}}=\infty.\label{H2E:blowupcond}
\end{align}
In particular if there exists $C$ such that
\begin{align}
\sum_{i=1}^n\n{\mathcal{F}_i(\bsym{u}(t))}_{\mathcal{X}_i}\leq C(\n{\bsym{u}(t)}_{\mathcal{X}^{\alpha}}+1) \ {\rm for} \  t\in[0,T_{\max})\label{H2E:sublinear}
\end{align}
then $T_{\max}=\infty$.
\end{lem}

\begin{proof}
If $n=1$ the result is well-known and can be proved using contraction mapping principle (see for instance [\cite{Lor}, Theorem 6.3.2]) . For $n>1$ one can adapt the same method with obvious modifications. 
\end{proof}

\subsection{Operators, semigroups, estimates}\label{H2E:secOperators}

\subsubsection{The $X^s$ spaces, operators $A_0,A_h$.}
Let us recall that $I_+=(0,1), \ I=(-1,1), \ \Omega=I\times I_+$. For $U\in\{I_+,I,\Omega\}$ we denote 
\begin{align*}
X(U)=L_2(U), \ (\cdot|\cdot)_{X(U)}=(\cdot|\cdot)_{L_2(U)}.
\end{align*}
For $i,j\in\mathbb{N}$ we define functions $u_i,v_i,w_{ij}$
\begin{align}
u_i(x_1)&=c_{1i}\cos(i\pi(x_1+1)/2), \ x_1\in I, \quad
v_i(x_2)=c_{2i}\cos(i\pi x_2), \ x_2\in I_+\\
w_{ij}(x_1,x_2)&=u_i(x_1)v_j(x_2), \ (x_1,x_2)\in\Omega\label{H2E:wij},
\end{align}
where constants $c_{1i},c_{2i}$ are such that $\n{u_i}_{X(I)}=\n{v_i}_{X(I_+)}=1$ i.e.
\begin{align}\label{H2E:cij}
c_{1i}=\begin{cases} 1/\sqrt{2} & \ {\rm if} \ i=0\\
									1 & \ {\rm if} \ i>0
\end{cases}, \quad 
c_{2i}=\begin{cases} 1 & \ {\rm if} \ i=0\\
									\sqrt{2} & \ {\rm if} \ i>0
\end{cases}.
\end{align}

The reason we introduce functions $u_i,v_i,w_{ij}$ is given in the following\\
\begin{lem}\label{H2E:basis}
The set $\{v_i: i\in\mathbb{N}\}$ (resp. $\{u_i: i\in\mathbb{N}\}$ and $\{w_{ij}: i,j\in\mathbb{N}\}$) is a complete orthonormal system in $X(I_+)$ (resp. $X(I)$ and $X(\Omega)$).
\end{lem}

\begin{proof}
The fact that  $\{v_i:  i\in\mathbb{N}\}$ is a complete orthonormal system in $X(I_+)$ is well known. Since $u_i(x)=(c_{1i}/c_{2i})v_i((x+1)/2)$ the thesis for the set $\{u_i:  i\in\mathbb{N}\}$ follows. Finally observe that since $w_{ij}=u_i\otimes v_j$ and $X(\Omega)=X(I)\otimes X(I_+)$ then the claim for the set $\{w_{ij}:  i,j\in\mathbb{N}\}$ follows from [\cite{Reed}, Chap. II.4, Prop. 2].

\begin{comment}
An easy computation shows that $(v_i|v_j)_{X(I_+)}=\delta_{ij}$. Thus it is left to prove that $\{v_i\}_{i\in\mathbb{N}}$ is linearly dense in $X(I_+)$. Recall that by the Stone-Weierstrass theorem set $V=\{\cos(\pi i x), \sin(\pi i x): \ i\in\mathbb{N}\}$ is linearly dense in $C(\overline{I})$. Thus due to continuous imbedding  $C(\overline{I})\subset X(I)$ set $V$ is also linearly dense in $X(I)$.

Choose $f\in X(I_+), \ \epsilon>0$. Using density of continuous functions in $X(I_+)$ choose $g\in C(\overline{I_+})$ such that $\n{f-g}_{C(\overline{I_+})}<\epsilon/2$. Define 
$g_{sym}(x)=\left\{\begin{array}{ll} g(x) & x\in [0,1]\\ g(-x) & x\in [-1,0]\end{array}\right.$ By   
\end{comment}
\end{proof}
Denote
\begin{align*}
X_{fin}(I)=lin(\{u_i:\ i\in\mathbb{N}\}), \quad X_{fin}(\Omega)=lin(\{w_{ij}:\ i,j\in\mathbb{N}\}).
\end{align*}
Define sequences
\begin{align*}
\lambda^{I_+}&=(\lambda^{I_+}_i)_{i\in\mathbb{N}}, \ 
\lambda^{I_+}_i=-(i\pi)^2, \ i\in\mathbb{N}\\
\lambda^I&=(\lambda^I_i)_{i\in\mathbb{N}}, \  \lambda^I_i=-(i\pi/2)^2, \ i\in\mathbb{N}\\
\lambda^{\Omega}_{,h}&=(\lambda^{\Omega}_{ij,h})_{i,j\in\mathbb{N}}, \ \lambda^{\Omega}_{ij,h}=\lambda_i^I+h^{-2}\lambda_j^{I_+}=-(i\pi/2)^2-(j\pi/h)^2, \ i,j\in\mathbb{N}, \ h\in(0,1]
\end{align*}
and denote $\lambda^{\Omega}=\lambda^{\Omega}_{,1}, \ \lambda^{\Omega}_{ij}=\lambda^{\Omega}_{ij,1}$. 
Next we define $X(I)$ and $X(\Omega)$ realisations of the perturbed Laplace operator witn Neumann boundary condition.
Define $\tilde{A_0}$ and  $\tilde{A_h}$ for $h\in(0,1]$ to be the unique unbounded linear operators such that
\begin{align*}
\tilde{A_0}&:X(I)\supset X_{fin}(I)\to X(I), \ \tilde{A_0}u_i=\partial^2_{x_1x_1}u_i=\lambda^I_iu_i,\\
\tilde{A_h}&:X(\Omega)\supset X_{fin}(\Omega)\to X(\Omega), \ \tilde{A_h}w_{ij}=-div J_h(w_{ij})=(\partial^2_{x_1x_1}+h^{-2}\partial^2_{x_2x_2})w_{ij}=\lambda^{\Omega}_{ij,h}w_{ij},
\end{align*}
and denote $\tilde{A}=\tilde{A_1}$.
\\
Define the unbounded linear operators $A_0$ and $A_h$ for $h\in(0,1]$:
\begin{align*}
A_0&:X(I)\supset D(A_0)\to X(I), \ D(A_0)=\{u\in X(I): \sum_{i\in\mathbb{N}}(1-\lambda_i^I)^2(u|u_i)_{X(I)}^2<\infty\}, \\  A_0u&=\sum_{i\in\mathbb{N}}\lambda_i^I(u|u_i)_{X(I)}u_i,\\
A_h&:X(\Omega)\supset D(A_h)\to X(\Omega), \ D(A_h)=\{w\in X(\Omega): \sum_{i,j\in\mathbb{N}}(1-\lambda_{ij,h}^{\Omega})^2(w|w_{ij})_{X(\Omega)}^2<\infty\},\\
A_hw&=\sum_{i,j\in\mathbb{N}}\lambda_{ij,h}^{\Omega}(w|w_{ij})_{X(\Omega)}w_{ij},
\end{align*}
and denote $A=A_1$. Observe that the domain $D(A_h)$ does not depend on $h$ since $\lambda^{\Omega}_{ij}\leq\lambda^{\Omega}_{ij,h}\leq h^{-2}\lambda^{\Omega}_{ij}$, i.e. $D(A_h)=D(A)$ for any $h\in(0,1]$.\\

\begin{comment}
\begin{lem}\label{H2E:closable}
For $h\in(0,1]$ the unbounded operators $\tilde{A_0}$ and $\tilde{A_h}$ are closable. 
\end{lem}

Denote the closure of $\tilde{A_0}$ (resp. $\tilde{A_h}$) by $A_0$ (resp. $A_h$) and put $A=A_1$.\\
\end{comment}

Next we collect spectral properties of operators $A_0, A_h$. 
\begin{lem}\label{H2E:selfadjoint}
\

\begin{enumerate}
\item Operator $A_0$ (resp. $A_h$) is the closure of operator $\tilde{A_0}$ (resp. $\tilde{A_h}$).
\item Operators $A_0$ and $A_h$  are self-adjoint and nonpositive.
\item
The spectra of operators $A_0, A_h$ consist entirely of eigenvalues:
\begin{align}
\sigma(A_0)=\sigma_p(A_0)=\lambda^I, \quad
\sigma(A_h)=\sigma_p(A_h)=\lambda^{\Omega}_{,h}.
\end{align}
\item Resolvent operators $R(\lambda,A_0)$ and $R(\lambda,A_h)$ satisfy
\begin{align}
R(\lambda,A_0)u&=\sum_{i\in\mathbb{N}}(\lambda-\lambda_i^I)^{-1}(u|u_i)_{X(I)}u_i, \ {\rm for} \ \lambda\in\rho(A_0), \ u\in X(I),\\
R(\lambda,A_h)w&=\sum_{i,j\in\mathbb{N}}(\lambda-\lambda_{ij,h}^{\Omega})^{-1}(w|w_{ij})_{X(\Omega)}w_{ij}, \ {\rm for} \ \lambda\in\rho(A_h), \ w\in X(\Omega).
\end{align}
\end{enumerate}
\end{lem}

\begin{proof}
We give the proof only for $A_h$ (for $A_0$ it is similar). Moreover it is clear that it is enough to consider the case $h=1$.\\
\textbf{Step 1} It is readilly seen that operator $A$ is an extension of operator $\tilde{A}$. To show that operator $A$ is closed let us consider an arbitrary sequence  $(w_n)_{n=1}^{\infty}\subset D(A)$ such that
\begin{align*}
w_n\to w, \ {\rm in} \ X(\Omega), \quad Aw_n\to v, \ {\rm in} \ X(\Omega),
\end{align*}
for certain $w,v\in X(\Omega)$. It follows that $(w_n)_{n=1}^{\infty}$ is a Cauchy sequence in the Hilbert space
$X^1(\Omega)=(D(A),(\cdot |\cdot)_{X^1(\Omega)})$, where
\begin{align*}
(w|w')_{X^1(\Omega)}=\sum_{i,j\in\mathbb{N}}(1-\lambda^{\Omega}_{ij})^2(w|w_{ij})_{X(\Omega)}(w'|w_{ij})_{X(\Omega)}, \ {\rm for} \ w,w'\in D(A).
\end{align*}
Thus $w\in D(A)$ and $Aw=v$ which proves that $A$ is closed. 
It is left to prove that $G(A)\subset cl_{X(\Omega)\times X(\Omega)}G(\tilde{A})$. To achieve this goal choose an arbitrary $(w,Aw)\in G(A)$. Then sequence $(w_n)_{n=1}^{\infty}$ defined by
$w_n=\sum_{i+j\leq n}(w|w_{ij})_{X(\Omega)}w_{ij}$ satisfies
\begin{align*}
w_n&\in D(\tilde{A}), \ {\rm for} \ n\geq 1,\\
(w_n,Aw_n)&\to (w,Aw), \ {\rm in} \ X(\Omega)\times X(\Omega),
\end{align*}
which completes the proof of 1.\\
\textbf{Step 2} To prove that operator $A$ is symmetric and nonpositive let us observe that for $w,w'\in D(A)$ we have
\begin{align*}
(Aw|w')_{X(\Omega)}&=\sum_{i,j\in\mathbb{N}}(Aw|w_{ij})_{X(\Omega)}(w'|w_{ij})_{X(\Omega)}=\sum_{i,j\in\mathbb{N}}\lambda_{ij}^{\Omega}(w|w_{ij})_{X(\Omega)}(w'|w_{ij})_{X(\Omega)}=\\
&=\sum_{i,j\in\mathbb{N}}(w|w_{ij})_{X(\Omega)}(Aw'|w_{ij})_{X(\Omega)}=(w|Aw')_{X(\Omega)},\\
(Aw|w)_{X(\Omega)}&=\sum_{i,j\in\mathbb{N}}\lambda_{ij}^{\Omega}(w|w_{ij})_{X(\Omega)}^2\leq 0.
\end{align*}
Moreover $A$ is densely defined since $X_{fin}(\Omega)\subset D(A)$ and $X_{fin}(\Omega)$ is dense in $X(\Omega)$ by Lemma \ref{H2E:basis}, thus it is possible to define the adjoint operator $A^*$. To prove that $A$ is self-adjoint it is left to prove that $D(A^*)=D(A)$, which is equivalent to $D(A^*)\subset D(A)$ since the opposite inclusion always holds. Choose arbitrary $w'\in D(A^*)$. By definition of $D(A^*)$ there exists unique $v\in X(\Omega)$ such that for every $w\in D(A)$ one has $(w'|Aw)_{X(\Omega)}=(v|w)_{X(\Omega)}$. Choosing $w=w_{ij}$ we obtain $\lambda_{ij}^{\Omega}(w'|w_{ij})_{X(\Omega)}=(v|w_{ij})_{X(\Omega)}$. Finally $w'\in D(A)$ since $\sum_{i,j}(1-\lambda_{ij}^{\Omega})^2(w'|w_{ij})^2_{X(\Omega)}\leq 2\sum_{i,j}(1+|\lambda_{ij}^{\Omega}|^2)(w'|w_{ij})^2_{X(\Omega)}=2(\n{w'}_{X(\Omega)}^2+\n{v}_{X(\Omega)}^2)<\infty$.\\
\textbf{Step 3} Since $Aw_{ij}=\lambda_{ij}^{\Omega}w_{ij}$ we obtain that $\lambda^{\Omega}\subset\sigma_p(A)$. For $\lambda\notin\lambda^{\Omega}$ define operator
\begin{align*}
B(\lambda,A)w=\sum_{i,j}(\lambda-\lambda_{ij}^{\Omega})^{-1}(w|w_{ij})_{X(\Omega)}w_{ij}.
\end{align*}
One checks easilly that $B(\lambda,A)\in\mathcal{L}(X(\Omega)),\ B(\lambda,A)w\in D(A)$ for $w\in X(\Omega)$. Moreover $\ B(\lambda,A)(\lambda-A)w=w$ for $w\in D(A)$ and $(\lambda-A)B(\lambda,A)w=w$ for $w\in X(\Omega)$ which is easilly seen for $w\in X_{fin}(\Omega)$ and by the density argument can be extended to $D(A)$ and $X(\Omega)$. Thus $\rho(A)=\mathbb{C}\setminus\lambda^{\Omega}, \  R(\lambda,A)=B(\lambda,A)$ for $\lambda\in\rho(A)$ and $\sigma(A)=\sigma_p(A)=\lambda^{\Omega}$.
\end{proof}

Since operators $A_0, A_h$ are self-adjoint and nonpositive they generate strongly continuous analytic semigroups  $e^{tA_0}$ and $e^{tA_h}:$
\begin{align}
e^{tA_0}u&=\sum_{i\in\mathbb{N}}e^{t\lambda_i^I}(u|u_i)_{X(I)}u_i, \ {\rm for} \ u\in X(I),\\
e^{tA_h}w&=\sum_{i,j\in\mathbb{N}}e^{t\lambda_{ij,h}^{\Omega}}(w|w_{ij})_{X(\Omega)}w_{ij}, \ {\rm for} \ w\in X(\Omega).
\end{align}

Since operators $I-A_0$ and $I-A$ are self-adjoint and positive one can define their fractional powers $(I-A_0)^s$ and $(I-A)^s$ for $s\geq 0$. Their domains $D((I-A_0)^s)$ and $D((I-A)^s)$ become Hilbert spaces (which we denote $X^s(I)$ and $X^s(\Omega)$) when equipped with appropriate scalar products. For $s\geq0$ spaces $X^s(I)$ and $X^s(\Omega)$ are defined as follows
\begin{align*}
X^s(I)&=\{u\in X(I): \sum_{i\in\mathbb{N}}(1-\lambda^I_i)^{2s}(u|u_i)_{X(I)}^2<\infty\},\\ 
(u|u')_{X^s(I)}&=\sum_{i\in\mathbb{N}}(1-\lambda_i^I)^{2s}(u|u_i)_{X(I)}(u'|u_i)_{X(I)}, \ {\rm for } \ u,u'\in X^s(I),\\
X^s(\Omega)&=\{w\in X(\Omega): \sum_{i,j\in\mathbb{N}}(1-\lambda^{\Omega}_{ij})^{2s}(w|w_{ij})_{X(\Omega)}^2<\infty\}, \\ 
(w|w')_{X^s(\Omega)}&=\sum_{i,j\in\mathbb{N}}(1-\lambda_{ij}^{\Omega})^{2s}(w|w_{ij})_{X(\Omega)}(w'|w_{ij})_{X(\Omega)}, \ {\rm for } \ w,w'\in X^s(\Omega).
\end{align*}

In the next lemma we give correspondence between scalar products in $X^s$ spaces.\\

\begin{lem}\label{H2E:lem6}
\
\begin{enumerate}
\item For $s_1>s_2\geq 0$ the following equalities hold
\begin{align}
(u|u_i)_{X^{s_1}(I)}&=(1-\lambda_i^I)^{2(s_1-s_2)}(u|u_i)_{X^{s_2}(I)}, \ {\rm for} \ u\in X^{s_1}(I),\label{H2E:lem61a}\\
(w|w_{ij})_{X^{s_1}(\Omega)}&=(1-\lambda_{ij}^{\Omega})^{2(s_1-s_2)}(w|w_{ij})_{X^{s_2}(\Omega)}, \ {\rm for} \ w\in X^{s_1}(\Omega).\label{H2E:lem61b}
\end{align} 
\item The set $\{u_i: i\in\mathbb{N}\}$ (resp. $\{w_{ij}: i,j\in\mathbb{N}\}$) is a complete orthogonal system in $X^s(I)$ (resp. $X^s(\Omega)$) for any $s\geq 0$. In particular if $s_1>s_2\geq0$ then $X^{s_1}(I)$ (resp. $X^{s_1}(\Omega)$) is a dense subspace of $X^{s_2}(I)$ (resp. $X^{s_2}(\Omega)$). 
\item For $U\in\{I,\Omega\}, \ s_1>s_2\geq0$ space $X^{s_1}(U)$ imbeds compactly into $X^{s_2}(U)$:
\begin{align}
X^{s_1}(U)\subset\subset X^{s_2}(U).\label{H2E:lem63}
\end{align}
\end{enumerate}
\end{lem}

\begin{proof}
\
\\
\textbf{Step 1} We give the proof for $U=I$ only as the one for $U=\Omega$ is analogous.
For $s\geq0$ and $u\in X^s(I)$ we have
\begin{align*}
(u|u_i)_{X^{s}(I)}=\sum_{k\in\mathbb{N}}(1-\lambda_k^I)^{2s}(u|u_k)_{X(I)}(u_i|u_k)_{X(I)}=(1-\lambda_i^I)^{2s}(u|u_i)_{X(I)},
\end{align*}
from which \eqref{H2E:lem61a} follows. The proof of \eqref{H2E:lem61b} is similar.\\
\textbf{Step 2} Orthogonality in $X^s$ follows from \eqref{H2E:lem61a},\eqref{H2E:lem61b} and Lemma \ref{H2E:basis} . For the proof of completeness of $\{u_i: i\in\mathbb{N}\}$ notice that if $u\in X^s(I)$ then denoting $u_n=\sum_{i=0}^n(u|u_i)_{X(I)}u_i$ one has
$\lim_{n\to\infty}\n{u-u_n}_{X^s(I)}^2=\lim_{n\to\infty}\sum_{k\geq n+1}(1-\lambda^I_i)^{2s}(u|u_k)^2_{X(I)}=0$. Similarly one prooves completeness of $\{w_{ij}: i,j\in\mathbb{N}\}$.\\
\textbf{Step 3}\\
Let $(u^n)_{n=1}^{\infty}$ be a bounded sequence in $X^{s_1}(I)$. Denote $M=\sup\{\n{u_n}_{X^{s_1}(I)}:\ n\in\mathbb{N}\}$. Since $X^{s_1}(I)$ is a Hilbert space we can choose a subsequence $(u^{n_k})_{k=1}$ weakly convergent in $X^{s_1}(I)$ to certain $u^{\infty}\in X^{s_1}(I)$. In particular
\begin{align*}
\lim_{k\to\infty}(u^{n_k}|u_i)_{X(I)}=(u^{\infty}|u_i)_{X(I)}, \ {\rm for} \  i\in\mathbb{N},\\
\n{u^{\infty}}_{X^{s_1}(I)}\leq M.
\end{align*}
For any $i_0\in\mathbb{N}_+$ we estimate
\begin{align*}
&\n{u^{n_k}-u^{\infty}}_{X^{s_2}(I)}^2\leq\sum_{i=0}^{i_0-1}(1-\lambda^I_i)^{2s_2}|(u^{n_k} -u^{\infty}| u_i)_{X(I)}|^2+(1-\lambda^I_{i_0})^{2s_2-2s_1}\sum_{i=i_0}^{\infty}(1-\lambda^I_i)^{2s_1}|(u^{n_k} -u^{\infty}|u_i)_{X(I)}|^2\\
&\leq\sum_{i=0}^{i_0-1}(1-\lambda^I_i)^{2s_2}|(u^{n_k} -u^{\infty} | u_i)_{X(I)}|^2+(1-\lambda^I_{i_0})^{2s_2-2s_1}\n{u^{n_k}-u^{\infty}}_{X^{s_1}(I)}^2\\
&\leq\sum_{i=0}^{i_0-1}(1-\lambda^I_i)^{2s_2}|(u^{n_k} -u^{\infty} | u_i)_{X(I)}|^2+4M^2(1-\lambda^I_{i_0})^{2s_2-2s_1}. 
\end{align*}
Fix $\epsilon>0$. Choose $i_0\in\mathbb{N}_+$ such that $4M^2(1-\lambda^I_{i_0})^{2s_2-2s_1}\leq\epsilon^2$. Then $\limsup_{k\to\infty}\n{u^{n_k}-u^{\infty}}_{X^{s_2}(I)}\leq\epsilon$ and consequently $\lim_{k\to\infty}\n{u^{n_k}-u^{\infty}}_{X^{s_2}(I)}=0$.
\end{proof}

Next we extend the scale of Hilbert spaces $X^s(I), X^s(\Omega)$ to  $s\in[-1,0)$ by duality. More precisely for any $s\in[-1,0)$ we define $X^s(I)=(X^{-s}(I))^*, X^s(\Omega)=(X^{-s}(\Omega))^*$. Then for $s\in[-1,0)$ Banach spaces $X^s$ become Hilbert spaces when equipped with the following scalar products
\begin{align}
(u|u')_{X^{s}(I)}=\sum_{i\in\mathbb{N}}(1-\lambda_i^I)^{2s}\Big<u,u_i\Big>_{(X^s(I),X^{-s}(I))}\Big<u',u_i\Big>_{(X^s(I),X^{-s}(I))}, \ {\rm for} \ u,u'\in X^s(I),\label{H2E:scalprodnegI}\\
(w|w')_{X^{s}(\Omega)}=\sum_{i,j\in\mathbb{N}}(1-\lambda_{ij}^{\Omega})^{2s}\Big<w,w_{ij}\Big>_{(X^s(\Omega),X^{-s}(\Omega))}\Big<w',w_{ij}\Big>_{(X^s(\Omega),X^{-s}(\Omega))}, \ {\rm for} \ w,w'\in X^s(\Omega)\label{H2E:scalprodnegO}.
\end{align}
Observe that assertions of Lemma \ref{H2E:lem6} are still valid without assumming that $s,s_1,s_2$ are nonnegative.\\

In the following lemma we give relation between $X^s$ spaces and complex interpolation.\\

\begin{lem}\label{H2E:reiteration}
$[X^{s_1}(U),X^{s_2}(U)]_{\theta}=X^{s_1(1-\theta)+s_2\theta}(U)$ for $s_1,s_2\geq-1, \ \theta\in[0,1], \ U\in\{I,\Omega\}$. 
\end{lem}

\begin{proof}
We provide the proof for $U=\Omega$ as the one for $U=I$ can be carried out similarly.
For $s\geq -1$ and $i,j\in\mathbb{N}$ define Hilbert spaces $Z_{ij}^s=(\mathbb{R},(\cdot|\cdot)_{Z_{ij}^s})$, where $(a|a')_{Z_{ij}^s}=(1-\lambda_{ij}^{\Omega})^{2s}aa'$ for $a,a'\in\mathbb{R}$ and $l_2(Z_{ij}^s)=(\{\mathbf{a}=(a_{ij})_{i,j\in\mathbb{N}}: a_{ij}\in\mathbb{R}, \ \sum_{i,j\in\mathbb{N}}\n{a_{ij}}_{Z_{ij}^s}^2<\infty\}, (\cdot |\cdot)_{l_2(Z_{ij}^s)})$, where $(\mathbf{a}|\mathbf{a'})_{l_2(Z_{ij}^s)}=\sum_{i,j\in\mathbb{N}}(a_{ij}|a_{ij}')_{Z_{ij}^s}$ for $\mathbf{a},\mathbf{a'}\in l_2(Z_{ij}^s)$.
Define map
\begin{align*}
\Phi(w)&=\Big(\Big<w,w_{ij}\Big>_{(X^{-1}(\Omega),X^1(\Omega))}\Big)_{i,j\in\mathbb{N}}, \ {\rm for} \ w\in X^{-1}(\Omega).
\end{align*}
Observe that $\Phi$ is an isometric  isomorphism between $X^s(\Omega)$ and $l_2(Z_{ij}^s)$ for any $s\geq-1$. This fact allows as to justify the first and the fourth equality in 
\begin{align*}
[X^{s_1}(\Omega),X^{s_2}(\Omega)]_{\theta}=[l_2(Z_{ij}^{s_1}),l_2(Z_{ij}^{s_2})]_{\theta}=l_2([Z_{ij}^{s_1},Z_{ij}^{s_2}]_{\theta})=l_2(Z_{ij}^{s_1(1-\theta)+s_2\theta})=X^{s_1(1-\theta)+s_2\theta}(\Omega),
\end{align*}
while the second equality follows from [\cite{Tri}, Chap. 1.18.1, Theorem].
\end{proof}

In the next lemma we characterise $X^s$ spaces as Sobolev-Slobodecki spaces with Neumann boundary condition.\\

\begin{lem}\label{H2E:characterisation}
For $s\in[0,1], \ U\in\{I,\Omega\}$ we have the following characterisation of the spaces $X^s(U)$:
\begin{align}
X^s(U)&=\begin{cases}W^{2s}_2(U) \ {\rm if}\ 0\leq s<3/4\\
W^{2s}_{2,N}(U)=\{u\in W^{2s}_2(U): \nabla u\cdot\nu=0 \ {\rm on} \ \partial U\} \ {\rm if}\ 3/4<s\leq1
\end{cases},
X^{3/4}(U)&\subset W^{2s}_2(U).
\end{align}
\end{lem}
 
\begin{proof}
The case when $U$ is an open bounded domain of $\mathbb{R}^n$ with a smooth boudary (in particular $U=I$) or $U$ is a half space - $U=\mathbb{R}_+\times\mathbb{R}^{n-1}$ was treated in [\cite{Fuj}, Theorem 2]. The case when $U=\Omega$ we divide in several steps.\\
\textbf{Step 1} We will show that 
\begin{align}
X^{1}(\Omega)=W^{2}_{2,N}(\Omega)\label{H2E:Step1goal}.
\end{align}
Denote
\begin{align*}
\overline{u}_i(x_1)=c_{1i}\sin(i\pi(x_1+1)/2), \ x_1\in I, \quad 
\overline{v}_i(x_2)=c_{2i}\sin(i\pi x_2), \ x_2\in I_+.
\end{align*}
Reasoning similarly as in the proof of  Lemma \ref{H2E:basis} we obtain that the set $\{\overline{u}_i : i\in\mathbb{N}_+\}$ (resp.  $\{\overline{v}_i : i\in\mathbb{N}_+\}$, $\{\overline{u}_i\otimes\overline{v}_j : i,j\in\mathbb{N}_+\}, \{\overline{u}_i\otimes v_j : i,\in\mathbb{N}_+, j\in\mathbb{N}\}, \{u_i\otimes\overline{v}_j : i\in\mathbb{N}_+,j\in\mathbb{N}\}$) is a complete orthonormal system in $X(I)$ (resp. $X(I_+), X(\Omega), X(\Omega), X(\Omega)$). Compute
\begin{align*}
\partial_{x_1}w_{ij}&=-(i\pi/2)\overline{u}_i\otimes v_j=-\sqrt{|\lambda^I_i|}\overline{u}_i\otimes v_j, \quad
&&\partial_{x_2}w_{ij}=-(j\pi)u_i\otimes\overline{v}_j=-\sqrt{|\lambda^{I_+}_j|}u_i\otimes\overline{v}_j\\
\partial^2_{x_1x_1}w_{ij}&=-(i\pi/2)^2w_{ij}=\lambda^I_iw_{ij}, \quad
&&\partial^2_{x_2x_2}w_{ij}=-(j\pi)^2w_{ij}=\lambda^{I_+}_j w_{ij}\\
\partial^2_{x_1x_2}w_{ij}&=\partial^2_{x_2x_1}w_{ij}=(i\pi/2)(j\pi)\overline{u}_i\otimes\overline{v}_j=\sqrt{|\lambda^{I}_i||\lambda^{I_+}_j|}\overline{u}_i\otimes\overline{v}_j.
\end{align*}
Observe that $X_{fin}(\Omega)\subset W^2_{2,N}(\Omega)$. Let $w\in X_{fin}(\Omega)$.  Using the  triangle inequality and $(a+b)^2\leq 2(a^2+b^2)$ we estimate
\begin{align*}
\n{w}_{X^1(\Omega)}^2=\n{(I-\Delta)w}_{L_2(\Omega)}^2\leq 2(\n{w}^2_{L_2(\Omega)}+2(\n{\partial^2_{x_1x_1}w}^2_{L_2(\Omega)}
+\n{\partial^2_{x_2x_2}w}^2_{L_2(\Omega)}))\leq 4\n{w}_{W^2_2(\Omega)}^2.
\end{align*}
On the other hand
\begin{align*}
\n{w}_{W^2_2(\Omega)}^2&=\n{w}_{L_2(\Omega)}^2
+\sum_{i=1}^2\n{\partial_{x_i}w}_{L_2(\Omega)}^2
+\sum_{i=1}^2\sum_{j=1}^2\n{\partial^2_{x_ix_j}w}_{L_2(\Omega)}^2\\
&=
\sum_{i,j}(1+|\lambda^I_i|+|\lambda^{I_+}_j|+|\lambda^I_i|^2+|\lambda^{I_+}_j|^2+2|\lambda^I_i||\lambda^{I_+}_j|)(w|w_{ij})^2_{X(\Omega)}\\
&\leq 2\sum_{i,j}(1-\lambda_i^I-\lambda_j^{I_+})^2(w|w_{ij})^2_{X(\Omega)}\\
&=2\sum_{i,j}(1-\lambda_{ij}^{\Omega})^2(w|w_{ij})^2_{X(\Omega)}=2\n{w}_{X^1(\Omega)}^2.
\end{align*}
Thus norms $\n{\cdot}_{W^2_2(\Omega)}$ and $\n{\cdot}_{X^1(\Omega)}$ are equivalent on $X_{fin}(\Omega)$.\\
In particular $X^1(\Omega)=cl_{X^1}(X_{fin}(\Omega))=cl_{W^2_2}(X_{fin}(\Omega))\subset W^2_{2,N}(\Omega)$. It is left to prove that $W^2_{2,N}(\Omega)\subset X^1(\Omega)$. Choose arbitrary $u\in W^2_{2,N}(\Omega)$ and let $f=u-\Delta u$. Then $f\in X(\Omega)$. Let $w=R(1,A)f$. Since $X^1(\Omega)\subset W^2_{2,N}(\Omega)$ thus $w\in W^2_{2,N}(\Omega)$ and $f=w-\Delta w$. We have  
\begin{align*}
0=\int_{\Omega}(w-u)(f-f)=\int_{\Omega}(w-u)^2-\int_{\Omega}(w-u)\Delta(w-u)=\int_{\Omega}(w-u)^2+\int_{\Omega}|\nabla(w-u)|^2,
\end{align*}
since $\nabla(w-u)\cdot\nu=0$ on $\partial\Omega$.
Finally we obtain that $u=w\in X^1(\Omega)$.
\\
\textbf{Step 2} We will show that for $U\in\{(\mathbb{R}_+)^2,\Omega\}$:
\begin{align}
[L_2(U),W^2_{2,N}(U)]_s
=\begin{cases}W^{2s}_2(U) \ {\rm if}\ 0\leq s<3/4\\
W^{2s}_{2,N}(U)\ {\rm if}\ 3/4<s\leq1.
\end{cases}\label{H2E:interp0}
\end{align}

To prove \eqref{H2E:interp0} for $U=(\mathbb{R}_+)^2$ we proceed as in the proof of [\cite{Fuj}, Theorem 2] for the case $U=\mathbb{R}_+\times\mathbb{R}$ substituting functions $\pi$ and $\nu$ from that proof by
\begin{align*}
\nu'&: L_2((\mathbb{R}_+)^2)\to L_2(\mathbb{R}^2), \  \nu'u(x_1,x_2)=u(|x_1|,|x_2|),\\
\pi'&:L_2(\mathbb{R}^2)\to L_2((\mathbb{R}_+)^2), \  \pi'v(x_1,x_2)=\frac{1}{4}\sum_{\epsilon_1,\epsilon_2\in\{-1,1\}}v(\epsilon_1x_1,\epsilon_2x_2).
\end{align*}

Observe that the only nonsmooth points of rectangle $\Omega$ are the corners. We choose the covering of $\Omega$ by four open subsets $\{\Omega_i\}$ such that each of them contains exactly one corner. Then a standard argument involving partition of unity inscribed in the covering $\{\Omega_i\}$ allows us to adapt \eqref{H2E:interp0} from $U=(\mathbb{R}_+)^2$ to $U=\Omega$.

\textbf{Step 3}
Using Lemma \ref{H2E:reiteration} and \eqref{H2E:Step1goal} we obtain $X^s(\Omega)=[X(\Omega),X^1(\Omega)]_s=[L_2(\Omega),W^2_{2,N}(\Omega)]_s$ from which Lemma \ref{H2E:characterisation} for $U=\Omega$ follows due to \eqref{H2E:interp0}.
\end{proof} 

In the next lemma we collect imbeddings of $X^s$ spaces into Lebesgue spaces $L_p$ and the space of continuous functions.\\  

\begin{lem}\label{H2E:imbeddings}
We have the following imbeddings
\begin{align}
X^s(I)
\subset\subset\begin{cases}C(\overline{I}) \ {\rm if}\ 1/4<s\\
L_p(I)\ {\rm if}\ 0\leq s<1/4, \ 1\leq p< 2/(1-4s)
\end{cases}
,X^{s}(I)\subset L_{2/(1-4s)}(I),
\label{H2E:imbedI}
\end{align}
\begin{align}
X^s(\Omega)
\subset\subset\begin{cases}C(\overline{\Omega}) \ {\rm if}\ 1/2<s\\
L_p(\Omega)\ {\rm if}\ 0\leq s<1/2, \ 1\leq p< 2/(1-2s)
\end{cases}
,X^{s}(\Omega)\subset L_{2/(1-2s)}(\Omega).
\label{H2E:imbedO}
\end{align}
\end{lem}

\begin{proof}
Imbeddings \eqref{H2E:imbedI} and \eqref{H2E:imbedO} are straigthforward consequences of the well-known continuous imbeddings of fractional Sobolev spaces $W^s_p$ (see for instance [\cite{Ada}, Theorem 7.27]), characterisation of $X^s$ spaces given in \ref{H2E:characterisation} and compact imbeddings of $X^s$ spaces given in \eqref{H2E:lem63}.
\end{proof}

For $s\geq -1$ define operator $A_{0,s}$ (resp. $A_{h,s}$) as $X^s(I)$ (resp. $X^s(\Omega)$) realisation of operator $A_0$ (resp. $A_h$) i.e.
\begin{align*}
A_{0,s}&:X^{s}(I)\supset X^{s+1}(I)\to X^{s}(I), \  A_{0,s}u=\sum_{i\in\mathbb{N}}\lambda_i^I(u|u_i)_{X(I)}u_i, \ {\rm for} \ u\in X^{s+1}(I),\\
A_{h,s}&:X^{s}(\Omega)\supset X^{s+1}(\Omega)\to X^{s}(\Omega), \  A_{h,s}w=\sum_{i,j\in\mathbb{N}}\lambda_{ij,h}^{\Omega}(w|w_{ij})_{X(\Omega)}w_{ij}, \ {\rm for} \ w\in X^{s+1}(\Omega).
\end{align*}

Operators $A_{0,s}, A_{h,s}$ are self-adjoint and nonpositive and thus generate strongly continuous, analytic semigroups of contractions $e^{tA_{0,s}}\in\mathcal{L}(X^{s}(I)),\  e^{tA_{h,s}}\in\mathcal{L}(X^s(\Omega))$.

If $s_1\geq s_2\geq-1$ then operators $A_{0,s_1},R(\lambda,A_{0,s_1}),e^{tA_{0,s_1}}$ are restrictions of operators $A_{0,s_2},R(\lambda,A_{0,s_2}),e^{tA_{0,s_2}}$ and operators $A_{h,s_1},R(\lambda,A_{h,s_1}),e^{tA_{h,s_1}}$ are restrictions of operators $A_{h,s_2},R(\lambda,A_{h,s_2}),e^{tA_{h,s_2}}$  i.e.
\begin{align*}
A_{0,s_1}u&=A_{0,s_2}u, \ {\rm for} \ u\in X^{s_1+1}(I),\\
R(\lambda,A_{0,s_1})u&=R(\lambda,A_{0,s_2})u,  \ 
e^{tA_{0,s_1}}u=e^{tA_{0,s_2}}u, \ {\rm for} \ u\in X^{s_1}(I),\\
A_{h,s_1}w&=A_{h,s_2}w, \ {\rm for} \ w\in X^{s_1+1}(\Omega),\\
R(\lambda,A_{h,s_1})w&=R(\lambda,A_{h,s_2})w,  \ 
e^{tA_{h,s_1}}w=e^{tA_{h,s_2}}w, \ {\rm for} \ w\in X^{s_1}(\Omega).
\end{align*}
From now on we will loose $s$-dependence in notation and write $A_0,A_h, R(\lambda,A_0), R(\lambda, A_h), e^{tA_0}, e^{tA_h}$ instead of $A_{0,s},A_{h,s},R(\lambda,A_{0,s}), R(\lambda, A_{h,s}), e^{tA_{0,s}}, e^{tA_{h,s}}$. \\

In the next lemma we collect basic estimates for the resolvents $R(\lambda,A_0), R(\lambda,A_h)$ and semigroups $e^{tA_0}, e^{tA_h}$.\\

\begin{lem}\label{H2E:sembasicest}
For $h\in(0,1], \ \lambda>0, \ t>0$ the following estimates hold
\begin{align}
\n{R(\lambda,A_{0})}_{\mathcal{L}(X^{s}(I),X^{s'}(I))}+\n{R(\lambda,A_h)}_{\mathcal{L}(X^{s}(\Omega),X^{s'}(\Omega))}\leq C\frac{1}{\lambda}(1+\lambda^{s'-s}), \ -1\leq s\leq s'\leq s+1\label{H2E:resest},\\
\n{e^{tA_{0}}}_{\mathcal{L}(X^{s}(I),X^{s'}(I))}+\n{e^{tA_h}}_{\mathcal{L}(X^{s}(\Omega),X^{s'}(\Omega))}\leq C\Big(1+\frac{1}{t^{s'-s}}\Big), \ -1\leq s\leq s'\label{H2E:semest},
\end{align}
where $C$ depends only on $s,s'$.
\end{lem}

\begin{proof}
The proof may be obtained with the use of spectral decomposition. For details we refer to proof of the Lemma \ref{H2E:lemiron} where we use the same technique.
\end{proof}

\subsubsection{Operators $E$, $P$ and $Tr$}
Define operators
\begin{align}
&E\in\mathcal{L}(X(I),X(\Omega)), \ [Eu](x_1,x_2)=u(x_1), \  {\rm for} \ u\in X(I),\label{H2E:defE}\\
&P\in\mathcal{L}(X(\Omega),X(I)), \ [Pw]
(x_1)=\int_{I_+}w(x_1,x_2)dx_2, \ {\rm for} \ w\in X(\Omega).\label{H2E:defP}
\end{align}
Basic properties of operators $E$ and $P$ are collected in the following\\

\begin{lem}\label{H2E:lemadj}
Operators $E$ and $P$ are mutually adjoint i.e. $E^*=P$. Moreover
\begin{align}
E\in\mathcal{L}(X^s(I),X^s(\Omega)), \ P\in\mathcal{L}(X^s(\Omega),X^s(I)), \ {\rm for} \ s\geq 0.\label{H2E:EPreg}
\end{align}

\end{lem}

\begin{proof}
To prove that $E^*=P$ we need to show that
\begin{align}
(Eu|w)_{X(\Omega)}=(u|Pw)_{X(I)}, \ {\rm for} \ u\in X(I), w\in X(\Omega)\label{H2E:condEP}.
\end{align}
Observe that for $i,j,k\in\mathbb{N}$ 
\begin{align}
Eu_k=w_{k0} \ {\rm and} \ Pw_{ij}=u_i\delta_{0j}.
\end{align}
Thus
\begin{align*}
(Eu_k|w_{ij})_{X(\Omega)}=(w_{k0}|w_{ij})_{X(\Omega)}=\delta_{ki}\delta_{0j}=(u_k|u_i\delta_{0j})_{X(I)}=(u_k|Pw_{ij})_{X(I)}.
\end{align*}
Owing to bilinearity of scalar products we obtain \eqref{H2E:condEP} for $u\in X_{fin}(I), w\in X_{fin}(\Omega)$ and finally by density of $X_{fin}(I)$ (resp. $X_{fin}(\Omega)$) in $X(I)$ (resp. $X(\Omega)$) and continuity of scalar products and operators $E, P$ we obtain \eqref{H2E:condEP} for arbitrary $u\in X(I), w\in X(\Omega)$.

\begin{comment}
Moreover for $u\in X(I), w\in X(\Omega)$
\begin{align*}
(Eu|w)_{X(\Omega)}=\int_{\Omega}u(x_1)w(x_1,x_2)=\int_I(u(x_1)\int_{I_+}w(x_1,x_2)dx_2)dx_1=
\end{align*}
\end{comment}

For $u\in X^s(I)$ we obtain
\begin{align*}
\n{Eu}_{X^s(\Omega)}^2&=\sum_{i,j}(1-\lambda_{ij}^{\Omega})^{2s}(Eu|w_{ij})^2_{X(\Omega)}=\sum_{i,j}(1-\lambda_{ij}^{\Omega})^{2s}(u|Pw_{ij})^2_{X(\Omega)}=\sum_{i,j}(1-\lambda_{ij}^{\Omega})^{2s}(u|u_i)^2_{X(\Omega)}\delta_{0j}\\
&=\sum_i(1-\lambda_{i0}^{\Omega})^{2s}(u|u_i)^2_{X(I)}=\n{u}_{X^s(I)}^2,
\end{align*}
since $\lambda^{\Omega}_{i0}=\lambda^I_i$.
Similarly
\begin{align*}
\n{Pw}_{X^s(I)}&=\sum_k(1-\lambda_k^I)^{2s}(Pw|u_k)^2_{X(I)}
=\sum_k(1-\lambda_k^I)^{2s}(w|Eu_k)^2_{X(I)}
\\
&=\sum_k(1-\lambda_{k0}^{\Omega})^{2s}(w|w_{k0})^2_{X(I)}\leq\n{w}_{X^s(\Omega)}^2.
\end{align*}
\end{proof}

Define operator $P_{-1}=E'$ and operator $E_{-1}=P'$. Using Lemma \ref{H2E:lemadj} we obtain that $P_{-1}$ and $E_{-1}$ satisfy
\begin{align*}   
P_{-1}\in\mathcal{L}(X^{-s}(\Omega),X^{-s}(I)), \ E_{-1}\in\mathcal{L}(X^{-s}(I),X^{-s}(\Omega)), \ s\in[0,1].
\end{align*}
Moreover for $u\in X(I), w\in X(\Omega)$
\begin{align*}
P_{-1}w=E'w=E^*w=Pw, \quad E_{-1}u=P'u=P^*u=Eu.
\end{align*}
From now on we will write $E,P$ instead of $E_{-1},P_{-1}$. 

\begin{comment}
extend operators $E,P$  to operators $E_{-1}\in\mathcal{L}(X^{-1}(I),X^{-1}(\Omega)), \ P_{-1}\in\mathcal{L}(X^{-1}(\Omega),X^{-1}(I))$ defined by
\begin{align*}
\Big<E_{-1}u,w\Big>_{(X^{-1}(\Omega)),X^1(\Omega))}&=\Big<u,Pw\Big>_{(X^{-1}(I),X^1(I))}, \ {\rm for} \ u\in X^{-1}(I), w\in X^1(\Omega)\\
\Big<P_{-1}w,u\Big>_{(X^{-1}(I),X^1(I))}&=\Big<w,Eu\Big>_{(X^{-1}(\Omega)),X^1(\Omega))},\ {\rm for} \ u\in X^{1}(I), w\in X^{-1}(\Omega).
\end{align*}
i.e. $E_{-1}=P'$ and $1$
From now on we will write $E,P$ instead of $E_{-1}, P_{-1}$.
\\
\end{comment}

For $w\in X_{fin}(\Omega)$ denote by $Tr(w)$ the trace operator i.e. restriction of $w$ to $I\times\{0\}$:
\begin{align*}
Tr(w)(x_1)=w(x_1,0), \ {\rm for} \ x_1\in I.
\end{align*}

\begin{lem}\label{H2E:tracelem}
For any $s>1/4$ there exists $C$ depending only on $s$ such that  for any $w\in X_{fin}(\Omega)$
\begin{align}
\n{Tr(w)}_{X^{s-1/4}(I)}\leq C\n{w}_{X^s(\Omega)} \label{H2E:traceest}.
\end{align}
Operator $Tr$ can be uniquely extended to an operator $\tilde{Tr}\in \mathcal{L}(X^{s}(\Omega),X^{s-1/4}(I))$ . 
\end{lem}

\begin{proof}
For $w=\sum_{i,j\geq0}a_{ij}w_{ij}$ where only finitely many $a_{ij}$ are nonzero we have
$Tr(w)=\sum_{i,j\geq0}a_{ij}u_iv_j(0)=\sum_{i\geq0}(\sum_{j\geq0}a_{ij}c_{2j})u_i$. Using Lemma \ref{H2E:lem6} we get that sytem $\{u_i\}$ is orthogonal in $X^{s-1/4}(I)$ and $\n{u_i}^2_{X^{s-1/4}(I)}=(1-\lambda_i^I)^{2s-1/2}=(1+(i\pi/2)^2)^{2s-1/2}$. Since $0<c_{2j}\leq\sqrt{2}$ (see \eqref{H2E:cij}) we thus obtain 
\begin{align*}
\n{Tr(w)}^2_{X^{s-1/4}(I)}&=\sum_{i\geq0}(\sum_{j\geq0}a_{ij}c_{2j})^2\n{u_i}^2_{X^{s-1/4}(I)}\leq 2\sum_{i\geq0}(\sum_{j\geq0}|a_{ij}|)^2(1+(i\pi/2)^2)^{2s-1/2}\\
&\leq 2\Big(\frac{\pi}{2}\Big)^{4s-1}\sum_{i\geq0}(\sum_{j\geq0}|a_{ij}|)^2(1+i)^{4s-1}.
\end{align*}
Using Cauchy-Schwarz inequality to estimate the inner sum we further obtain that
\begin{align}
\n{Tr(w)}^2_{X^{s-1/4}(I)}&\leq 2\Big(\frac{\pi}{2}\Big)^{4s-1}\sum_{i\geq0}(\sum_{j\geq0}|a_{ij}|^2(1+i+j)^{4s})(\sum_{j\geq0}\frac{1}{(1+i+j)^{4s}})(1+i)^{4s-1}\nonumber \\
&\leq  2\Big(\frac{\pi}{2}\Big)^{4s-1}\frac{4s}{4s-1}\sum_{i\geq0}(\sum_{j\geq0}|a_{ij}|^2(1+i+j)^{4s})\label{H2E:tracelem1},
\end{align}
where the last inequality is a consequence of the following estimate
\begin{align*}
\sum_{j\geq0}\frac{(1+i)^{4s-1}}{(1+i+j)^{4s}}&=\frac{1}{1+i}+\sum_{j\geq1}\frac{(1+i)^{4s-1}}{(1+i+j)^{4s}}\leq 1+(1+i)^{4s-1}\int_{1+i}^{\infty}\frac{dt}{t^{4s}}\\
&=1+(1+i)^{4s-1}\frac{1}{4s-1}(1+i)^{1-4s}=\frac{4s}{4s-1}.
\end{align*}

On the other hand since system $\{w_{ij}\}$ is orthogonal in $X^s(\Omega)$ and 
\begin{align*}
\n{w_{ij}}^2_{X^s(\Omega)}=(1-\lambda_{ij}^{\Omega})^{2s}=(1+(i\pi/2)^2+(j\pi)^2)^{2s}\geq 3^{-2s}(1+i+j)^{4s}
\end{align*}
 we have
\begin{align}
&\n{w}_{X^s(\Omega)}^2=\sum_{i,j\geq0}|a_{ij}|^2\n{w_{ij}}^2_{X^s(\Omega)}=\sum_{i\geq0}(\sum_{j\geq0}|a_{ij}|^2(1+(i\pi/2)^2+(j\pi)^2)^{2s})\nonumber\\
&\geq 3^{-2s}\sum_{i\geq0}(\sum_{j\geq0}|a_{ij}|^2(1+i+j)^{4s})\label{H2E:tracelem2}.
\end{align}
Combining \eqref{H2E:tracelem1} and \eqref{H2E:tracelem2} we obtain \eqref{H2E:traceest} with $C^2=3^{2s}(\pi/2)^{4s-1}8s/(4s-1)$. Since $X_{fin}(\Omega)$ is dense in $X^{s}(\Omega)$ (see part 2 of Lemma \ref{H2E:lem6})  the latter part of the Lemma \ref{H2E:tracelem} follows.
\end{proof}

From now on we write $Tr$ instead of $\tilde{Tr}$.\\
Next we collect several identities involving operators $P,E,Tr,R(\lambda,A_h), R(\lambda,A_0), e^{tA_h}$ and $e^{tA_0}$.\\
\begin{lem}
The following identities hold
\bs\label{H2E:Identities}
\eq{
PTr'u=TrEu=PEu&=u, \ {\rm for} \ u\in X^{-s}(I), \ 1\geq s>0 \label{H2E:Iden0}\\
R(\lambda,A_h)E&=ER(\lambda,A_{0}), \ {\rm for} \ h>0 \label{H2E:Iden1} \\ 
e^{tA_h}E&=Ee^{tA_{0}}, \ {\rm for} \  h>0 \label{H2E:Iden2}.
}
\es
\end{lem}

\begin{proof}
Identities $TrEu=PEu=u$ are obvious for $u\in X(I)$ and can be extended to the case when $u\in X^{-s}(I)$ by a density argument. Then
\begin{align*}
PTr'u=E'Tr'u=(TrE)'u=u,
\end{align*}
from which \eqref{H2E:Iden0} follows.\\
Since $Eu_i=w_{i0}$ for any $i\geq0$ hence
\begin{align*} 
R(\lambda,A_h)Eu_i=R(\lambda,A_h)w_{i0}=\frac{1}{\lambda-\lambda_{i0}^{\Omega}}w_{i0}=E\frac{1}{\lambda-\lambda_i^{I}}u_i=ER(\lambda,A_{0})u_i.
\end{align*}
Since $\{u_i\}_{i\geq0}$ is a Schauder basis in every $X^s(I)$ (see part 2 of Lemma \ref{H2E:lem6}) we obtain \eqref{H2E:Iden1}. Similarly one proves \eqref{H2E:Iden2}. 
\end{proof}

\subsubsection{Resolvent and semigroup estimates for dimension reduction}
Estimates for semigroup $e^{tA_h}$ and resolvent operator $R(\lambda,A_h)$ which are presented in the next lemma are of fundamental importance in the dimension reduction carried out in section \ref{H2E:secdimred}.\\
\begin{lem}\label{H2E:lemiron}
For $h\in(0,1], \ s,s'\geq-1, \ t,\lambda>0,\ w\in X^s(\Omega)$ the following estimates hold
\begin{align}
&\n{R(\lambda,A_h)(I-EP)w}_{X^{s'}(\Omega)}\leq \frac{1}{\lambda-\lambda_{01,h}^{\Omega}}(1+(\lambda-\lambda_{01,h}^{\Omega})^{s'-s})\n{(I-EP)w}_{X^{s}(\Omega)}, \ 0\leq s'-s\leq 1,\label{H2E:resestiron}\\
&\n{e^{tA_h}(I-EP)w}_{X^{s'}(\Omega)}\leq C\Big(1+\frac{1}{t^{s'-s}}\Big)e^{t\lambda_{01,h}^{\Omega}}\n{(I-EP)w}_{X^{s}(\Omega)}, \ 0\leq s'-s.\label{H2E:semestiron}
\end{align}
where $C$ depends only on $s,s'$.
\end{lem}

\begin{proof}
Since $X_{fin}(\Omega)$ is dense in $X^{s}(\Omega)$ (see part 2 of Lemma \ref{H2E:lem6}) one can assume that $w\in X_{fin}(\Omega)$ i.e. $w=\sum_{i,j\geq0}a_{ij}w_{ij}$ where only finitely many $a_{ij}$ are nonzero. Define
\begin{align*}
M_1=\sup\Big\{\frac{(1-\lambda^{\Omega}_{ij})^{s'-s}}{\lambda-\lambda^{\Omega}_{ij,h}}: \ i\geq0, j\geq1 \Big\}.
\end{align*} 
Observe that  
\begin{align}
&(I-EP)w_{ij}=w_{ij}-E(u_{i}\delta_{0j})=w_{ij}-w_{i0}\delta_{0j}=w_{ij}(1-\delta_{0j}),\label{H2E:clever1}\\
&(w_{ij}|w_{kl})_{X^{s'}(\Omega)}=(1-\lambda_{ij}^{\Omega})^{2s'}\delta_{ij}\delta_{jl},\label{H2E:clever2}\\
&\n{w_{ij}}_{X^{s'}(\Omega)}=(1-\lambda_{ij}^{\Omega})^{s'-s}\n{w_{ij}}_{X^s(\Omega)}\label{H2E:clever3}.
\end{align}
Indeed \eqref{H2E:clever1} is a simple consequence of the definitions of operators $E$ and $P$ (see \eqref{H2E:defE},\eqref{H2E:defP}) and \eqref{H2E:wij}, while \eqref{H2E:clever2} and \eqref{H2E:clever3} follow from Lemma \ref{H2E:lem6}.
Using \eqref{H2E:clever1}, \eqref{H2E:clever2} and \eqref{H2E:clever3} we estimate
\begin{align*}  
&\n{R(\lambda,A_h)(I-EP)w}_{X^{s'}(\Omega)}^2=\n{\sum_{i\geq0,j\geq1}\frac{1}{\lambda-\lambda_{ij,h}^{\Omega}}a_{ij}w_{ij}}^2_{X^{s'}(\Omega)}=\sum_{i\geq0,j\geq1}\frac{1}{(\lambda-\lambda_{ij,h}^{\Omega})^2}a_{ij}^2\n{w_{ij}}_{X^{s'}(\Omega)}^2\\
&=\sum_{i\geq0,j\geq1}\frac{1}{(\lambda-\lambda_{ij,h}^{\Omega})^2}a_{ij}^2(1-\lambda_{ij}^{\Omega})^{2(s'-s)}\n{w_{ij}}_{X^s(\Omega)}^2\leq M_1^2\sum_{i\geq0,j\geq1}a_{ij}^2\n{w_{ij}}_{X^s(\Omega)}^2\\
&=M_1^2\n{(I-EP)w}^2_{X^s(\Omega)},
\end{align*}
To finish the proof of \eqref{H2E:resestiron} it is left to show that
\begin{align}
M_1\leq\frac{1}{\lambda-\lambda^{\Omega}_{01,h}}(1+(\lambda-\lambda^{\Omega}_{01,h})^{s'-s})\label{H2E:resestironpr}.
\end{align}
Using condition $0\leq s'-s\leq 1$ and the following inequality
\begin{align*}
(1+x)^{\alpha}\leq 1+x^{\alpha} \ {\rm for} \ x>0,  \ 0\leq \alpha\leq1 
\end{align*}
we estimate
\begin{align*}
&\frac{(1-\lambda^{\Omega}_{ij})^{s'-s}}{\lambda-\lambda^{\Omega}_{ij,h}}\leq\frac{(1-\lambda^{\Omega}_{ij,h})^{s'-s}}{\lambda-\lambda^{\Omega}_{ij,h}}\leq\frac{1+(-\lambda^{\Omega}_{ij,h})^{s'-s}}{\lambda-\lambda^{\Omega}_{ij,h}}\leq\frac{1+(\lambda-\lambda^{\Omega}_{ij,h})^{s'-s}}{\lambda-\lambda^{\Omega}_{ij,h}}\\
&=\frac{1}{(\lambda-\lambda^{\Omega}_{ij,h})^{1-(s'-s)}}+\frac{1}{\lambda-\lambda^{\Omega}_{ij,h}}\leq\frac{1}{(\lambda-\lambda^{\Omega}_{01,h})^{1-(s'-s)}}+\frac{1}{\lambda-\lambda^{\Omega}_{01,h}}=\frac{1}{\lambda-\lambda^{\Omega}_{01,h}}(1+(\lambda-\lambda^{\Omega}_{01,h})^{s'-s})
\end{align*}
from which \eqref{H2E:resestironpr} and consequently \eqref{H2E:resestiron} follows. We move to the proof of \eqref{H2E:semestiron}. Reasoning as in the proof of \eqref{H2E:resestiron} we obtain that for $w\in X^s(\Omega)$
\begin{align*}
\n{e^{tA_h}(w-EPw)}_{X^{s'}(\Omega)}\leq M_2\n{w-EPw}_{X^s(\Omega)},
\end{align*}
where 
\begin{align*}
M_2=\sup\{(1-\lambda^{\Omega}_{ij})^{s'-s}\exp(t\lambda^{\Omega}_{ij,h}): i\geq0, j\geq1\}.
\end{align*}
Using inequality \eqref{H2E:ineq1} from Lemma \ref{H2E:ineq} we estimate for $i\geq0, j\geq 1$
\begin{align*}
&(1-\lambda^{\Omega}_{ij})^{s'-s}\exp(t\lambda^{\Omega}_{ij,h})=(1+(i\pi/2)^2+(j\pi)^2)^{s'-s}\exp(-t((i\pi/2)^2+(j\pi/h)^2))\\
&=(1+(i\pi/2)^2+(j\pi)^2)^{s'-s}\exp(-\frac{t}{h^2}(1+(i\pi/2)^2+(j\pi)^2))\exp(-t(i\pi/2)^2)\exp(\frac{t}{h^2}(1+(i\pi/2)^2))\\
&\leq\sup\{x^{s'-s}\exp(-\frac{t}{h^2}x): x\geq 1+(i\pi/2)^2+\pi^2\}\exp(-t(i\pi/2)^2)\exp(\frac{t}{h^2}(1+(i\pi/2)^2))\\
&\leq C((\frac{h^2}{t})^{s'-s}+(1+(i\pi/2)^2+\pi^2)^{s'-s})\exp(-\frac{t}{h^2}(1+(i\pi/2)^2+\pi^2)) \exp(-t(i\pi/2)^2)\exp(\frac{t}{h^2}(1+(i\pi/2)^2))\\
&=C((\frac{h^2}{t})^{s'-s}+(1+(i\pi/2)^2+\pi^2)^{s'-s})\exp(-t(i\pi/2)^2)\exp(-\frac{t\pi^2}{h^2})\\
&\leq C(\frac{1}{t^{s'-s}}+1+\frac{(t(i\pi/2)^{2})^{s'-s}}{t^{s'-s}})\exp(-t(i\pi/2)^2)\exp(-\frac{t\pi^2}{h^2})\\
&\leq C(1+\frac{1}{t^{s'-s}}+\frac{1}{t^{s'-s}}\sup\{x^{s'-s}\exp(-x): x\geq0\})\exp(-\frac{t\pi^2}{h^2})\leq C(1+\frac{1}{t^{s'-s}})\exp(-\frac{t\pi^2}{h^2}).
\end{align*}

\end{proof}

\subsubsection{The multiplication operator}

For $1\leq p<\infty$ and $0\geq f\in L_p(I)$ we define the multiplication operator $M_f$
\begin{align}
M_f:L_{\infty}(I)\supset D(M_f)\to L_{\infty}(I), \ M_fu=fu,
\end{align}
where $D(M_f)=\{u\in L_{\infty}(I): fu\in L_{\infty}(I)\}$.
Observe that if $u\in L_{\infty}(I)$ and $Re(\lambda)>0$ then
$R(\lambda,M_f)u=\frac{u}{\lambda-f}\in L_{\infty}(I)$ and $\n{R(\lambda,M_f)}_{\mathcal{L}(L_{\infty}(I))}\leq 1/|\lambda|$, which proves that $M_f$ is sectorial and thus generates an analytic semigroup $e^{tM_f}$:
\begin{align*}
e^{tM_f}u=e^{tf}u, \ u\in L_{\infty}(I).
\end{align*}
Basic estimates concerning $e^{tM_f}$ are collected in the following\\
\begin{lem}\label{H2E:Multi}
Assume that $0\geq f,f_1,f_2\in L_p(I)$. Then for $t,t'\geq0$ 
\begin{align}
\n{e^{tM_f}}_{\mathcal{L}(L_{\infty}(I))}&\leq1,\label{H2E:Multi0}\\
\n{e^{t'M_f}-e^{tM_f}}_{\mathcal{L}(L_{\infty}(I),L_{p}(I))}&\leq |t'-t|\n{f}_{L_p(I)},\label{H2E:Multi1}\\
\n{e^{tM_{f_1}}-e^{tM_{f_2}}}_{\mathcal{L}(L_{\infty}(I),L_{p}(I))}&\leq t\n{f_1-f_2}_{L_p(I)}\label{H2E:Multi2}. 
\end{align} 
\end{lem}

\begin{proof}
Using inequalities
\begin{align*}
0<e^x\leq 1, \ |e^{x}-e^{y}|\leq|x-y|, \ x,y<0,
\end{align*}
we get for $u\in L_{\infty}(I), t,t'\geq0$
\begin{align*}
\n{e^{tM_f}u}_{\infty}&=\n{e^{tf}u}_{L_{\infty}(I)}\leq\n{e^{tf}}_{\infty}\n{u}_{L_{\infty}(I)}\leq\n{u}_{\infty},\\
\n{(e^{t'M_f}-e^{tM_f})u}_{L_p(I)}&=\n{(e^{t'f}-e^{tf})u}_{L_p(I)}\leq\n{e^{t'f}-e^{tf}}_{L_p(I)}\n{u}_{\infty}\leq |t'-t|\n{f}_{L_p(I)}\n{u}_{\infty},\\
\n{(e^{tM_{f_1}}-e^{tM_{f_2}})u}_{L_p(I)}&=\n{(e^{tf_1}-e^{tf_2})u}_{L_p(I)}\leq\n{e^{tf_1}-e^{tf_2}}_{L_p(I)}\n{u}_{\infty}\leq t\n{f_1-f_2}_{L_p(I)}\n{u}_{\infty},
\end{align*}
from which \eqref{H2E:Multi0}, \eqref{H2E:Multi1} and \eqref{H2E:Multi2} follow.
\end{proof}

\section{The case of a regular source}\label{H2E:secregular}

Denote
\begin{align*}
\boldsymbol{u}&=(u_1,u_2,u_3,u_4,u_5)\in\mathbb{R}^5, \ 
f_1,f_2,f_3,f_4,f_5:\mathbb{R}^5\to\mathbb{R},\\
f_1(\boldsymbol{u})&=-(c_1+u_3)u_1+c_2u_2+c_4u_4,\\
f_2(\boldsymbol{u})&=c_1u_1-(b_2+c_2+c_3u_3)u_2+c_5u_5,\\
f_3(\boldsymbol{u})&=-(b_3+u_1+c_{3}u_2)u_3+c_4u_4+c_5u_5+p_3,\\
f_4(\boldsymbol{u})&=u_1u_3-(b_4+c_4)u_4,\\
f_5(\boldsymbol{u})&=c_3u_2u_3-(b_5+c_5)u_5.
\end{align*}

In this section we study system \eqref{H2E:System} with $\delta$ substituted by a regular function $\omega$:
\bs\label{H2E:SystemReg}
\eq{
\partial_t u_1+div(J_h(u_1))+b_1u_1&=0,&& (t,x)\in\Omega_{T}\label{H2E:SystemRegA}\\
\partial_t u_2-d\partial^2_{x_1} u_2&=f_2(\boldsymbol{u}),&&(t,x)\in(\partial_1\Omega)_{T}\label{H2E:SystemRegB}\\
\partial_t u_3&=f_3(\boldsymbol{u}),&&(t,x)\in(\partial_1\Omega)_{T}\label{H2E:SystemRegC}\\
\partial_t u_4&=f_4(\boldsymbol{u}),&&(t,x)\in(\partial_1\Omega)_{T}\label{H2E:SystemRegD}\\
\partial_t u_5&=f_5(\boldsymbol{u}),&&(t,x)\in(\partial_1\Omega)_{T}\label{H2E:SystemRegE}
}
\es
with boundary and initial conditions
\begin{align*}
-J_h(u_1)\nu&=0,&&(t,x)\in(\partial_0\Omega)_{T}\\
-J_h(u_1)\nu&=f_1(\boldsymbol{u})+\omega,&&(t,x)\in(\partial_1\Omega)_{T}\\
\partial_{x_1} u_2&=0,&&(t,x)\in(\partial\partial_1\Omega)_{T}\\
\boldsymbol{u}(0,\cdot)&=\boldsymbol{u_0}
\end{align*}

To obtain well-posedness of system \eqref{H2E:SystemReg} we interpret it as a system of abstract ODE's \eqref{H2E:abs1}-\eqref{H2E:abs2}.

Assume that
\bs\label{H2E:Sysucond}
\eq{
&d, \boldsymbol{b}>0,\  \boldsymbol{c},\boldsymbol{p}\geq 0,\\
&1/2<s<s'<3/4,\label{H2E:sass}\\
&0\leq\omega\in L_{\infty}(I).
}
\es

Define spaces
\begin{align*}
\mathcal{X}_1&=X^{-1+s'}(\Omega), \mathcal{X}_2=X^0(I), \mathcal{X}_3=\mathcal{X}_4=\mathcal{X}_5=L_{\infty}(I)\\
\mathcal{X}_1^1&=X^{s'}(\Omega), \mathcal{X}_2^1=X^1(I), \mathcal{X}_3^1=\mathcal{X}_4^1=\mathcal{X}_5^1=L_{\infty}(I)
\end{align*}
Set $\alpha=(1+s-s',1/2,1/2,1/2,1/2)$ and observe that due to Lemma \ref{H2E:reiteration} we have 
\begin{align*}
\mathcal{X}^{\alpha}=X^s(\Omega)\times X^{1/2}(I)\times (L_{\infty}(I))^3.
\end{align*}
Define operators
\begin{align*}
\mathcal{A}_1u&=(A_h-b_1)u, \ u\in \mathcal{X}_1^1,\\
\mathcal{A}_2u&=dA_{0}u, \ u\in \mathcal{X}_2^1\\
\mathcal{A}_i&=0, \ i=3,4,5
\end{align*}
and for $\bsym{u}\in\mathcal{X}^{\alpha}$ set
\begin{align*} 
\mathcal{F}_1(\bsym{u})&=Tr'[f_1(Tr(u_1),u_2,\ldots,u_5)+\omega]\\
\mathcal{F}_i(\bsym{u})&=f_i(Tr(u_1),u_2,\ldots,u_5), \ i=2,3,4,5.
\end{align*}
The main result of the present section is the following\\

\begin{theo}\label{H2E:existencereg}
Assume \eqref{H2E:Sysucond}. Then for every $0\leq\bsym{u}_0\in\mathcal{X}^{\alpha}$ system \eqref{H2E:SystemReg} has a unique globally in time defined $\mathcal{X}^{\alpha}$ solution $\bsym{u}$. The solution $\bsym{u}$ is nonnegative and satisfies for all times the following estimate
\begin{align}
\n{u_3(t)}_{\infty}+\n{u_4(t)}_{\infty}+\n{u_5(t)}_{\infty}\leq C,\label{H2E:linftybound}
\end{align}
where $C$ depends only on $\n{u_{30}}_{\infty}+\n{u_{40}}_{\infty}+\n{u_{50}}_{\infty}, b_3,b_4,b_5,p_3$.
\end{theo}

\begin{proof}
\
\\
\textbf{Step 1 - local existence of solution.}
\

Using assumption \eqref{H2E:sass}, Lemma \ref{H2E:imbeddings} and Lemma \ref{H2E:tracelem} we get
\begin{align*}
\mathcal{X}^{\alpha}&=X^s(\Omega)\times X^{1/2}(I)\times (L_{\infty}(I))^3\subset C(\overline{\Omega})\times C(\overline{I})\times (L_{\infty}(I))^3,\\
Tr'&\in\mathcal{L}(L_{\infty}(I),\mathcal{X}_1),\\
Tr&\in\mathcal{L}(\mathcal{X}_1^{\alpha_1},C(\overline{I}))
\end{align*}
from where we deduce that for $\bsym{u},\bsym{w}\in\mathcal{X}^\alpha$ the following estimates hold
\begin{align}
&\sum_{i=1}^5\n{\mathcal{F}_i(\bsym{u})}_{\mathcal{X}_i}\leq C\Big\{(1+\sum_{i=1}^2\n{u_i}_{\mathcal{X}_i^{\alpha_i}})(1+\n{u_3}_{\mathcal{X}_3^{\alpha_3}})+\sum_{i=4}^5\n{u_i}_{\mathcal{X}_i^{\alpha_i}}+\n{\omega}_{\infty}\Big\}\label{H2E:Fsublinear}\\
&\sum_{i=1}^5\n{\mathcal{F}_i(\bsym{u})-\mathcal{F}_i(\bsym{u}')}_{\mathcal{X}_i^{\alpha_i}}\leq C\Big\{\sum_{i=1}^2\n{u_i-u_i'}_{\mathcal{X}_i^{\alpha_i}}(1+\n{u_3}_{\mathcal{X}_3^{\alpha_3}}+\n{u_3'}_{\mathcal{X}_3^{\alpha_3}})\nonumber\\
&+\n{u_3-u_3'}_{\mathcal{X}_3^{\alpha_3}}(1+\sum_{i=1}^2(\n{u_i}_{\mathcal{X}_i^{\alpha_i}}+\n{u_i'}_{\mathcal{X}_i^{\alpha_i}}))+\sum_{i=4}^5\n{u_i-u_i'}_{\mathcal{X}_i^{\alpha_i}}\Big\}\nonumber
\end{align}

Using above estimates we conclude that assumptions of Lemma \ref{H2E:exlem} are satisfied which results in the existence of a unique maximally defined $\mathcal{X}^{\alpha}$ solution to \eqref{H2E:SystemReg}. 
\
\newline
\\
\textbf{Step 2 - nonnegativity of solution.}
\

Reasoning as in Step 1 we obtain that system 
\bs\label{H2E:SystemRegPlus}
\eq{
\partial_t v_1+div(J_h(v_1))+b_1v_1&=0,&& (t,x)\in\Omega_{T}\label{H2E:SystemRegPlusA}\\
\partial_t v_2-d\partial^2_{x_1} v_2&=f_{2+}(\boldsymbol{v}),&&(t,x)\in(\partial_1\Omega)_{T}\label{H2E:SystemRegPlusB}\\
\partial_t v_3&=f_{3+}(\boldsymbol{v}),&&(t,x)\in(\partial_1\Omega)_{T}\label{H2E:SystemRegPlusC}\\
\partial_t v_4&=f_{4+}(\boldsymbol{v}),&&(t,x)\in(\partial_1\Omega)_{T}\label{H2E:SystemRegPlusD}\\
\partial_t v_5&=f_{5+}(\boldsymbol{v}),&&(t,x)\in(\partial_1\Omega)_{T}\label{H2E:SystemRegPlusE}
}
\es
with boundary and initial conditions
\begin{align*}
-J_h(v_1)\nu&=0,&&(t,x)\in(\partial_0\Omega)_{T}\\
-J_h(v_1)\nu&=f_{1+}(\boldsymbol{v})+\omega,&&(t,x)\in(\partial_1\Omega)_{T}\\
\partial_{x_1} v_2&=0,&&(t,x)\in(\partial\partial_1\Omega)_{T}\\
\boldsymbol{v}(0,\cdot)&=\boldsymbol{u_0}
\end{align*}

where for $i=1,\ldots, 5$ and $\bsym{v}\in\mathbb{R}^5$
\begin{align*}
f_{i+}(\bsym{v})=f_i((v_1)_{+},\ldots,(v_5)_{+})
\end{align*}
has a unique maximal $\mathcal{X}^{\alpha}$ solution $\bsym{v}(t)$. Note by $T_{\max}'$ its time of existence.\\
Testing \eqref{H2E:SystemRegPlusA},$\ldots$,\eqref{H2E:SystemRegPlusE} by $(v_1)_{-},\ldots,(v_5)_{-}$ we obtain 
\begin{align*}
&-\frac{1}{2}\frac{d}{dt}\n{(v_1)_{-}}^2_{X(\Omega)}-
\n{\partial_{x_1}(v_1)_{-}}^2_{X(\Omega)}-h^{-2}\n{\partial_{x_2}(v_1)_{-}}^2_{X(\Omega)}-b_1\n{(v_1)_-}^2_{X(\Omega)}=\\
&\int_I(f_{1+}(v_1(x_1,0),v_2(x_1),\ldots,v_5(x_1))+\omega(x_1))(v_1(x_1,0))_{-}dx_1\\
&-\frac{1}{2}\frac{d}{dt}\n{(v_2)_{-}}^2_{X(I)}-d\n{\partial_{x_1}(v_2)_{-}}^2_{X(I)}=\int_If_{2+}(v_1(x_1,0),v_2(x_1),\ldots,v_5(x_1))(v_2(x_1))_{-}dx_1\\
&-\frac{1}{2}\frac{d}{dt}\n{(v_i)_{-}}^2_{X(I)}=\int_If_{i+}(v_1(x_1,0),v_2(x_1),\ldots,v_5(x_1))(v_i(x_1))_{-}dx_1, \ i=3,4,5.
\end{align*}
Since right hand sides of above equalities are nonnegative we obtain that
\begin{align*}
&\frac{d}{dt}[\n{(v_1)_{-}}^2_{X(\Omega)}+\sum_{i=2}^5\n{(v_i)_{-}}^2_{X(I)}]\leq0\\
&\n{(v_1(t))_{-}}^2_{X(\Omega)}+\sum_{i=2}^5\n{(v_i(t))_{-}}^2_{X(I)}\leq \n{(v_{01})_{-}}^2_{X(\Omega)}+\sum_{i=2}^5\n{(v_{0i})_{-}}^2_{X(I)}=0.
\end{align*}
Which proves that the only solution of system \eqref{H2E:SystemRegPlus} is nonnegative. Since for $\boldsymbol{v}\geq0$ there is $f_{i+}(\bsym{v})=f_{i}(\bsym{v})$ we see that $T_{\max}\geq T_{\max}'$ and $\bsym{u}(t)=\bsym{v}(t)$ for $t\in[0,T_{max}')$. Finally observe that if $T_{\max}'<\infty$ then owing to the blow-up condition \eqref{H2E:blowupcond}
\begin{align*}
\limsup_{t\to {T_{\max}'}^{-}}\n{\bsym{u}(t)}_{\mathcal{X}^{\alpha}}=\limsup_{t\to {T_{\max}'}^{-}}\n{\bsym{v}(t)}_{\mathcal{X}^{\alpha}}=\infty
\end{align*}
whence $T_{\max}=T_{\max}'$ and finally $\bsym{u}(t)\geq0$ for $t\in[0,T_{\max})$.

\textbf{Step 3 - global solvability: $T_{\max}=\infty$.}
\

Adding equations \eqref{H2E:SystemRegC},\eqref{H2E:SystemRegD},\eqref{H2E:SystemRegE} and using nonnegativity of $\bsym{u}$ we get
\begin{align*}
\partial_t(u_3+u_4+u_5)+\min\{b_3,b_4,b_5\}(u_3+u_4+u_5)\leq p_3
\end{align*}
from which we conclude that there exists $C$ depending only on $\n{u_{30}}_{\infty}+\n{u_{40}}_{\infty}+\n{u_{50}}_{\infty}, b_3,b_4,b_5,p_3$ such that
\begin{align}
\n{u_3(t)}_{\infty}+\n{u_4(t)}_{\infty}+\n{u_5(t)}_{\infty}\leq C,\ t\in[0,T_{\max})\label{H2E:linftybound2}.
\end{align}
Using \eqref{H2E:linftybound2} and \eqref{H2E:Fsublinear} we get that condition \eqref{H2E:sublinear} is satisfied which gives $T_{\max}=\infty$.

\begin{comment}
\begin{align*}
&|<\mathcal{F}_1(\bsym{u}),v>|\leq \n{f_0(\bsym{u})}_{X(\Omega)}\n{v}_{X(\Omega)}+(\n{f_1(\bsym{u}(x_1,0))}_{X(I)}+\n{g}_{X(I)})\n{v(x_1,0)}_{X(I)}\\
&\leq C[\n{u_1}_{X(\Omega)}+\n{u_1(x_1,0)}_{X(I)}(\n{u_3}_{\infty}+1)+\n{u_2}_{X(I)}+\n{u_4}_{X(I)}+\n{g}_{X(I)}]\n{v}_{X^{1-s'}(\Omega)}\\
&\leq C[\n{u_1}_{\mathcal{X}_1^{\theta_1}}(\n{u_3}_{\mathcal{X}_3^{\theta_3}}+1)+\n{u_2}_{\mathcal{X}_2^{\theta_2}}+\n{u_4}_{\mathcal{X}_4^{\theta_4}}+\n{g}_{X(I)}]\n{v}_{X^{1-s'}(\Omega)}\\
\
&|<\mathcal{F}_1(\bsym{u})-\mathcal{F}_1(\bsym{w}),v>|\leq\n{f_0(\bsym{u})-f_0(\bsym{w})}_{X(\Omega)}\n{v}_{X(\Omega)}+\n{f_1(\bsym{u}(x_1,0))-f_1(\bsym{w}(x_1,0))}_{X(I)}\n{v(x_1,0)}_{X(I)}\\
&\leq C[\n{u_1-w_1}_{X(\Omega)}+\n{u_1-w_1}_{X(I)}(1+\n{u_3}_{X(I)}+\n{w_3}_{X(I)})+\n{u_3-w_3}_{X(I)}(\n{u_1}_{X(I)}+\n{w_1}_{X(I)})\\
&+\n{u_2-w_2}_{X(I)}+\n{u_4-w_4}_{X(I)}]\n{v}_{X^{1-s'}(\Omega)}\\
&\leq C[\n{u_1-w_1}_{\mathcal{X}_1^{\theta_1}}(1+\n{u_3}_{\mathcal{X}_3^{\theta_3}}+\n{w_3}_{\mathcal{X}_3^{\theta_3}})+\n{u_3-w_3}_{\mathcal{X}_3^{\theta_3}}(\n{u_1}_{\mathcal{X}_1^{\theta_1}}+\n{w_1}_{\mathcal{X}_1^{\theta_1}})\\
&+\n{u_2-w_2}_{\mathcal{X}_2^{\theta_2}}+\n{u_4-w_4}_{\mathcal{X}_4^{\theta_4}}]\n{v}_{X^{1-s'}(\Omega)}\\
\
\end{align*}
\end{comment}

\end{proof}

\section{The case of a singular source and dimension reduction}\label{H2E:secdimred}

We begin by introducing auxiliary function which are used in the definition of M-mild solution presented in section \ref{H2E:secM-mild}.

\subsection{Auxiliary functions}\label{H2E:secAuxiliary}

Let us recall the definition of the standard one dimensional mollifier
\begin{align*}
\eta(x_1)=\begin{cases} C\exp\Big(\frac{1}{|x_1|^2-1}\Big), \quad |x_1|<1\\
0, \quad |x_1|\geq1
\end{cases},\quad
\eta^{\epsilon}(x_1)=\eta(x_1/\epsilon)/\epsilon, \ \epsilon>0
\end{align*}
where $C$ is such that $\int_{\mathbb{R}}\eta=1$. 
\\

The next lemma concerns convergence of $\eta^{\epsilon}$ as $\epsilon\to0$.\\
\begin{lem}\label{H2E:deltacon}
For any $0<s$ the following convergence holds
\begin{align}
\lim_{\epsilon\to0^+}\n{\eta^{\epsilon}-\delta}_{X^{-1/4-s}(I)}=0.\label{H2E:deltaconv}
\end{align}
\end{lem}

\begin{proof}
Without loss of generality assume that $s<1/8$. It is enough to show that every sequence $(\epsilon_n)_{n=1}^{\infty}$ of positive numbers which converges to $0$ has a subsequence $(\epsilon_{n_k})_{k=1}^{\infty}$ such that 
\begin{align}
\eta^{\epsilon_{n_k}}\to\delta \ {\rm in} \  X^{-1/4-s}(I).\label{H2E:auxlemgoal}
\end{align}
Since $X^{1/4+s}(I)\subset\subset C(\overline{I})$ (see Lemma \ref{H2E:imbeddings}) thus $\mathcal{M}(\ov{I})=C(\overline{I})^*\subset\subset X^{-1/4-s}(I)$. Fix any sequence $(\epsilon_n)_{n=1}^{\infty}$ of positive numbers which converges to $0$. 
Since $(\eta^{\epsilon_n})_{n=1}^{\infty}$ is a bounded sequence in $\mathcal{M}(\ov{I})$ then, by the previous observation, one can choose a subsequence $(\epsilon_{n_k})_{k=1}^{\infty}$ such that 
\begin{align*}
\eta^{\epsilon_{n_k}}\to u \ {\rm in} \  X^{-1/4-s}(I),
\end{align*}
for certain $u\in X^{-1/4-s}(I)$. Finally observe that for any $v\in X^{1/4+s}(I)$ one has
\begin{align*}
\Big<u,v\Big>_{(X^{-1/4-s}(I),X^{1/4+s}(I))}=\lim_{k\to\infty}\Big<\eta^{\epsilon_{n_k}},v\Big>_{(X^{-1/4-s}(I),X^{1/4+s}(I))}=\lim_{k\to\infty}\int_I\eta^{\epsilon_{n_k}}v=v(0),
\end{align*}
where the first equality is a consequence of the fact that strong convergence in $X^{-1/4-s}(I)$ implies convergence in the weak star topology of $X^{-1/4-s}(I)$ while the third equality follows from a well known fact that $\eta^{\epsilon}$ converges to $\delta$ in the weak star topology of $\mathcal{M}(\ov{I})$. Thus $u=\delta$ and \eqref{H2E:auxlemgoal} follows.
\end{proof}

%Owing to Lemma \eqref{deltacon} we put $\eta^0=\delta$. \\
From now on we denote
\begin{align}
\eta^0=\delta, \ \mu=(h,\epsilon)\in(0,1]\times[0,1] \ {\rm and } \ \mu_0=(h,0).
\end{align}
Next we define auxilliary functions $m^{\mu}$ and $m^0$ which play a fundamental role in the definition of M-mild solution which is given in section \ref{H2E:secM-mild}:
\begin{align}
m^{\mu}=R(b_1,A_h)(p_1Tr'\eta^{\epsilon}),\ m^{0}=R(b_1,A_{0})(p_1\delta).\label{H2E:auxdef}
\end{align}

From \eqref{H2E:auxdef} we get that $m^{\mu}$ for $\epsilon>0$ and $m^0$ are $W^1_2$ weak solutions of the following boundary value problems
\bs\label{H2E:bvp}
\eq{
b_1m^{\mu}+div(J_h(m^{\mu}))&=0, && x\in\Omega\\
-J_h(m^{\mu})\nu&=0, && x\in\partial_0\Omega\\
-J_h(m^{\mu})\nu&=p_1\eta^{\epsilon}, && x\in\partial_1\Omega,
}
\es
\bs\label{H2E:bvp0}
\eq{
b_1m^{0}-d\partial^2_{x_1}m^{0}&=p_1\delta, && x_1\in I\\
\partial_{x_1}m^{0}&=0, && x_1\in\partial I.
}
\es

Concernig regularity of $m^0$ and $m^{\mu}$ we have the following \\

\begin{lem}\label{H2E:swallow}
Let $m^0$ and $m^{\mu}$ be given by \eqref{H2E:auxdef} then 
\begin{align}
m^0&\in W^1_{\infty}(I),\label{H2E:regstat0}\\
m^{\mu}&\in W^1_p(\Omega) \ {\rm for} \ {\rm any} \ 1\leq p <2,\label{H2E:regstat}\\
\n{m^{\mu}}_{X^{1/2-s}(\Omega)}&\leq C, \ 0<s\leq3/2 \label{H2E:swallow0} \\
\n{m^{\mu}-m^{\mu_0}}_{X^{1/2-s}(\Omega)}&\leq C\n{\eta^{\epsilon}-\delta}_{X^{-1/4-s}(I)}, \ 0<s\leq3/4 \label{H2E:swallow1} \\
\n{m^{\mu_0}-Em^{0}}_{X^{1/2-s}(\Omega)}&\leq C\frac{1}{|\lambda_{01,h}^{\Omega}|^{s/2}}, \ 0<s\leq3/2 \label{H2E:swallow2},
\end{align}
where $C$ does not depend on $\mu$. Moreover $m^{\mu},m^{0}\geq0$.
\end{lem}

\begin{proof}
To prove \eqref{H2E:regstat0} define $u(x_1)=m^0(x_1)+\frac{p_1}{2d}|x_1|$. Then using \eqref{H2E:bvp0} we obtain that $b_1u-d\partial^2_{x_1}u=\frac{b_1p_1}{2d}|x_1|$ for $x_1\in I$. We conclude that $u\in C^2(\ov{I})$ from where \eqref{H2E:regstat0} follows. The claim \eqref{H2E:regstat} is a consequence of [\cite{Mal4}, Lemma 1]. \\
Using \eqref{H2E:resest},  \eqref{H2E:traceest}, \eqref{H2E:deltaconv} we estimate
\begin{align*}
\n{m^{\mu}}_{X^{1/2-s}(\Omega)}&\leq p_1\n{R(b_1,A_h)}_{\mathcal{L}(X^{-1/2-s}(\Omega),X^{1/2-s}(\Omega))}\n{Tr'}_{\mathcal{L}(X^{-1/4-s}(I),X^{-1/2-s}(\Omega))}\n{\eta^{\epsilon}}_{X^{-1/4-s}(I)}\leq C,
\end{align*}
from which \eqref{H2E:swallow0} follows.
To prove \eqref{H2E:swallow1} we proceed in a similar manner:
\begin{align*}
\n{m^{\mu}-m^{\mu_0}}_{X^{1/2-s}(\Omega)}&\leq p_1\n{R(b_1,A_h)}_{\mathcal{L}(X^{-1/2-s}(\Omega),X^{1/2-s}(\Omega))}\n{Tr'}_{\mathcal{L}(X^{-1/4-s}(I),X^{-1/2-s}(\Omega))}\n{\eta^{\epsilon}-\delta}_{X^{-1/4-s}(I)}.
\end{align*}
Using \eqref{H2E:Iden0} we get that $\delta=PTr'\delta$ hence using \eqref{H2E:Iden1} and \eqref{H2E:resestiron} we obtain: 
\begin{align*}
&\n{m^{\mu_0}-Em^{0}}_{X^{1/2-s}(\Omega)}=p_1\n{R(b_1,A_h)Tr'\delta-ER(b_1,A_{0})PTr'\delta}_{X^{1/2-s}(\Omega)}\\
&=p_1\n{R(b_1,A_h)(I-EP)Tr'\delta)}_{X^{1/2-s}(\Omega)}\leq C\Big(\frac{1}{(b_1-\lambda_{01,h}^{\Omega})^{s/2}}+\frac{1}{b_1-\lambda_{01,h}^{\Omega}}\Big)\n{(I-EP)Tr'\delta}_{X^{-1/2-s/2}(\Omega)}.
\end{align*}
Moreover using \eqref{H2E:traceest}, \eqref{H2E:deltaconv} we have 
\begin{align*}
\n{(I-EP)Tr'\delta}_{X^{-1/2-s/2}(\Omega)}\leq\n{I-EP}_{\mathcal{L}(X^{-1/2-s/2}(\Omega))}\n{Tr'}_{\mathcal{L}(X^{-1/4-s/2}(I),X^{-1/2-s/2}(\Omega))}\n{\delta}_{X^{-1/4-s/2}(I)}\leq C.
\end{align*}
Finally to finish the proof of \eqref{H2E:swallow2} observe that
\begin{align*}
\frac{1}{(b_1-\lambda_{01,h}^{\Omega})^{s/2}}+\frac{1}{b_1-\lambda_{01,h}^{\Omega}}\leq C\frac{1}{|\lambda_{01,h}^{\Omega}|^{s/2}}.
\end{align*}
Using maximum principle for elliptic boundary value problem \eqref{H2E:bvp} we get that $m^{\mu}\geq0$ for $\epsilon>0$. Then \eqref{H2E:swallow1} implies that $m^{\mu_0}\geq0$ while $m^0\geq0$ follows from \eqref{H2E:swallow2}.
\end{proof}

Recall that $\mu=(h,\epsilon)\in(0,1]\times[0,1]$. Substituting $\delta$ by $\eta^{\epsilon}$ in \eqref{H2E:System}  we get 
\bs\label{H2E:SystemApu}
\eq{
\partial_t u_{1}^{\mu}+div(J_h(u_{1}^{\mu}))+b_1u_{1}^{\mu}&=0,&& (t,x)\in\Omega_{T}\label{H2E:SystemApuA}\\
\partial_t u_{2}^{\mu}-d\partial^2_{x_1} u_{2}^{\mu}&=f_2(\boldsymbol{u}^{\mu}),&&(t,x)\in(\partial_1\Omega)_{T}\label{H2E:SystemApuB}\\
\partial_t u_{3}^{\mu}&=f_3(\boldsymbol{u}^{\mu}),&&(t,x)\in(\partial_1\Omega)_{T}\label{H2E:SystemApuC}\\
\partial_t u_{4}^{\mu}&=f_4(\boldsymbol{u}^{\mu}),&&(t,x)\in(\partial_1\Omega)_{T}\label{H2E:SystemApuD}\\
\partial_t u_{5}^{\mu}&=f_5(\boldsymbol{u}^{\mu}),&&(t,x)\in(\partial_1\Omega)_{T}\label{H2E:SystemApuE}
}
\es
with boundary and initial conditions
\begin{align}
-J_h(u_{1}^{\mu})\nu&=0,&&(t,x)\in(\partial_0\Omega)_{T}\nonumber\\
-J_h(u_{1}^{\mu})\nu&=f_1(\boldsymbol{u}^{\mu})+p_1\eta^{\epsilon},&&(t,x)\in(\partial_1\Omega)_{T}\label{H2E:SystemApuBCD}\\
\partial_{x_1} u_{2}^{\mu}&=0,&&(t,x)\in(\partial\partial_1\Omega)_{T}\nonumber\\
\boldsymbol{u}^{\mu}(0,\cdot)&=\boldsymbol{u}_{0}\nonumber,
\end{align}
where
\begin{align*}
\bsym{u}^{\mu}=(u_{1}^{\mu},u_{2}^{\mu},u_{3}^{\mu},u_{4}^{\mu},u_{5}^{\mu}).
\end{align*}

\subsection{Definition of M-mild solution}\label{H2E:secM-mild}

Using Theorem \ref{H2E:existencereg} we obtain that for $\epsilon\in(0,1]$ system \eqref{H2E:SystemApu} has a unique globally defined $\mathcal{X}^{\alpha}$ solution. Unfortunately due to regularity issues the notion of $\mathcal{X}^{\alpha}$ solution is insufficient for the case $\epsilon=0$. Any potential solution $u_{1}^{\mu_0}$ has to be unbounded function of $x$ for any positive time which causes problems in the ODE part of the system. This motivates us to generalize the notion of solution. We rewrite our problem in the new variables so that system \eqref{H2E:SystemApu} with singular source term is transformed into system \eqref{H2E:SystemApz} with regular sources and low regularity initial data.  

Observe that  putting 
\begin{align}
\bsym{z}^{\mu}=(z_{1}^{\mu},z_{2}^{\mu},z_{3}^{\mu},z_{4}^{\mu},z_{5}^{\mu})=M(u_{1}^{\mu}-m^{\mu},u_{2}^{\mu},u_{3}^{\mu},u_{4}^{\mu},u_{5}^{\mu}),\label{H2E:zmitrans} \\
\bsym{z}_{0}^{\mu}=(z_{01}^{\mu},z_{02},z_{03},z_{04},z_{05})=M(u_{01}-m^{\mu},u_{02},u_{03},u_{04},u_{05})\label{H2E:zmi0trans},
\end{align}
where $m^{\mu}$ was defined in \eqref{H2E:auxdef} and $M$ denotes the following matrix
\begin{align*}
M=\begin{bmatrix}1,0,0,0,0\\0,1,0,0,0\\0,0,1,0,0\\0,0,1,1,0\\0,0,1,1,1\end{bmatrix},
\end{align*}
system \eqref{H2E:SystemApu} can be rewritten as

\bs\label{H2E:SystemApz}
\eq{
\partial_t z_{1}^{\mu}+div(J_h(z_{1}^{\mu}))+b_1z_{1}^{\mu}&=0,&& (t,x)\in\Omega_{T}\label{H2E:SystemApzA}\\
\partial_t z_{2}^{\mu}-d\partial^2_{x_1} z_{2}^{\mu}&=g_{2}^{\mu}(\bsym{z}^{\mu}),&&(t,x_1)\in(\partial_1\Omega)_{T}\label{H2E:SystemApzB}\\
\partial_t z_{3}^{\mu}+Tr(m^{\mu})z_{3}^{\mu}&=g_3(\boldsymbol{z}^{\mu}),&&(t,x_1)\in(\partial_1\Omega)_{T}\label{H2E:SystemApzC}\\
\partial_t z_{4}^{\mu}&=g_4(\boldsymbol{z}^{\mu}),&&(t,x_1)\in(\partial_1\Omega)_{T}\label{H2E:SystemApzD}\\
\partial_t z_{5}^{\mu}&=g_5(\boldsymbol{z}^{\mu}),&&(t,x_1)\in(\partial_1\Omega)_{T}\label{H2E:SystemApzE}
}
\es
with boundary and initial conditions
\begin{align*}
-J_h(z_{1}^{\mu})\nu&=0,&&(t,x)\in(\partial_0\Omega)_{T}\\
-J_h(z_{1}^{\mu})\nu&=g_{1}^{\mu}(\bsym{z}^{\mu}),&&(t,x_1)\in(\partial_1\Omega)_{T}\\
\partial_{x_1} z_{2}^{\mu}&=0,&&(t,x_1)\in(\partial\partial_1\Omega)_{T}\\
\boldsymbol{z}^{\mu}(0,\cdot)&=\boldsymbol{z}_{0}^{\mu},
\end{align*}
where 
\begin{align*}
g_{1}^{\mu},g_{2}^{\mu}&:I\times\mathbb{R}^5\to\mathbb{R},\ g_{3},g_{4},g_{5}:\mathbb{R}^5\to\mathbb{R},\\
g_{1}^{\mu}(\boldsymbol{z})&=-c_1z_1+c_2z_2-z_1z_3+c_4(z_4-z_3)-(c_1+z_3)Tr(m^{\mu}),\\
g_{2}^{\mu}(\boldsymbol{z})&=-b_2z_2+c_1z_1-c_2z_2-c_3z_2z_3+c_5(z_5-z_4)+c_1Tr(m^{\mu}),\\
g_{3}(\boldsymbol{z})&=-b_3z_3-z_1z_3-c_3z_2z_3+c_4(z_4-z_3)+c_5(z_5-z_4)+p_3,\\
g_{4}(\boldsymbol{z})&=-b_3z_3-b_4(z_4-z_3)-c_3z_2z_3+c_5(z_5-z_4)+p_3,\\
g_{5}(\boldsymbol{z})&=-b_3z_3-b_4(z_4-z_3)-b_5(z_5-z_4)+p_3.
\end{align*}

Assume that:
\bs\label{H2E:asssing}
\eq{
&d,\bsym{b}>0, \ \bsym{c}, \bsym{p}\geq0, \label{H2E:asssing1}\\
&2<p<\infty, \ 0<\theta<\min\Big\{\frac{1}{16},\frac{1}{2p}\Big\},\label{H2E:asssing2}\\
&0\leq \bsym{u}_0=(u_{01},\ldots, u_{05})\in X^{1/2+\theta}(\Omega)\times X^{1/2}(I)\times\{L_{\infty}(I)\}^3\label{H2E:asssing3}.
}
\es

Define Banach spaces
\begin{align*}
\bsym{Z}_{-}&=Z_{1-}\times Z_{2-}\times Z_{3-}\times Z_{4-}\times Z_{5-}=X^{-1/4-\theta}(\Omega)\times X(I)\times L_p(I)\times L_p(I)\times L_p(I),\\
\bsym{Z}&=Z_1\times Z_2\times Z_3\times Z_4\times Z_5=X^{1/2-\theta}(\Omega)\times X^{1/2}(I)\times L_p(I)\times L_p(I)\times L_p(I),\\
\bsym{Z}_{+}&=Z_{1+}\times Z_{2+}\times Z_{3+}\times Z_{4+}\times Z_{5+}=X^{1/2+\theta}(\Omega)\times X^{1/2}(I)\times L_{\infty}(I)\times L_{\infty}(I)\times L_{\infty}(I).
\end{align*}
For $\bsym{z}\in \bsym{Z_{+}}$ put
\begin{align*}
G_{1}^{\mu}(\bsym{z})&=Tr'(g_{1}^{\mu}(Tr(z_1),z_2,z_3,z_4,z_5)), \\
G_{2}^{\mu}(\bsym{z})&=g_{2}^{\mu}(Tr(z_1),z_2,z_3,z_4,z_5), \\
G_{i}(\bsym{z})&=g_i(Tr(z_1),z_2,z_3,z_4,z_5), \ i\in\{3,4,5\}.
\end{align*}

\begin{defi}
Fix $\mu=(h,\epsilon)\in(0,1]\times[0,1]$ and let $\bsym{z}^{\mu}, \bsym{z}_{0}^{\mu}$ be related with $\bsym{u}^{\mu}, \bsym{u}_{0}$  by equations \eqref{H2E:zmitrans} and \eqref{H2E:zmi0trans}. We define $\bsym{u}^{\mu}$ as a \textbf{M-mild} solution of system \eqref{H2E:SystemApu} on $[0,T)$ if the following three conditions are satisfied
\begin{enumerate} 
 \item Assumptions \eqref{H2E:asssing} hold. 
\item Function $\bsym{z}^{\mu}$ has the following regularity
\bs\label{H2E:deficond}
\eq{
z_{1}^{\mu}&\in C([0,T),Z_1), \ t^{2\theta}z_{1}^{\mu}\in L_{\infty}(0,T';Z_{1+}) \ {\rm for} \ T'<T,\label{H2E:deficond1}\\
z_{2}^{\mu}&\in C([0,T),Z_2),\\
z_{3}^{\mu}&\in C([0,T),Z_3)\cap L_{\infty}(0,T;Z_{3+}),\\
z_{i}^{\mu}&\in C([0,T),Z_{i+}), \ i\in\{4,5\}.
}
\es
\item For every $t\in[0,T)$ the following Duhamel formulas hold
\bs\label{H2E:Duhamelz}
\eq{
z_{1}^{\mu}(t)&=e^{t(A_h-b_1)}z_{01}^{\mu}+\int_0^te^{(t-\tau)(A_h-b_1)}G_{1}^{\mu}(\bsym{z}^{\mu}(\tau))d\tau,\\
z_{2}^{\mu}(t)&=e^{tdA_{0}}z_{02}+\int_0^te^{(t-\tau)dA_{0}}G_{2}^{\mu}(\bsym{z}^{\mu}(\tau))d\tau,\\
z_{3}^{\mu}(t)&=e^{-tTr(m^{\mu})}z_{03}+\int_0^te^{-(t-\tau)Tr(m^{\mu})}G_3(\bsym{z}^{\mu}(\tau))d\tau,\\
z_{i}^{\mu}(t)&=z_{0i}+\int_0^t G_i(\bsym{z}^{\mu}(\tau))d\tau, \ i\in\{4,5\}.
}
\es
\end{enumerate} 
\end{defi}

Concerning regularity of M-mild solutions we have the following\\

\begin{rem}
If $\bsym{u}^{\mu}$ is a M-mild solution of system \eqref{H2E:SystemApu} then 
\bs\label{H2E:M-mildreg}
\eq{
&u_1^{\mu}\in C([0,T),W^{1-2\theta}_2(\Omega)),  \ t^{2\theta}u_{1}^{\mu}\in L_{\infty}(0,T';W^1_p(\Omega)) \ {\rm for} \ 1\leq p<2,\\
&u_2^{\mu}\in C([0,T),W^1_2(I))\label{H2E:M-mildreg2},\\
&u_i\in C([0,T),L_p(I))\cap L_{\infty}(0,T';L_{\infty}(I)) \ {\rm for} \ i\in\{3,4,5\}\label{H2E:M-mildreg3}.
}
\es
\end{rem}

\begin{proof}
Using Lemma \ref{H2E:characterisation} we obtain that $Z_1=W^{1-2\theta}_2(\Omega), Z_{1+}=W^{1+2\theta}_2(\Omega), Z_2=W^1_2(I)$. Using Lemma \ref{H2E:swallow} $m^{\mu}\in W^1_p(\Omega)\cap W^{1-2\theta}_2(\Omega)$ for $1\leq p<2$. Thus using \eqref{H2E:zmitrans} and \eqref{H2E:deficond} we obtain that 
\begin{align*}
&u_1^{\mu}=z_1^{\mu}+m^{\mu}\subset C([0,T),W^{1-2\theta}_2(\Omega)),\\
&t^{2\theta}u^{\mu}_1=t^{2\theta}z^{\mu}_1+t^{2\theta}m^{\mu}\subset L_{\infty}(0,T';W^1_2(\Omega))+L_{\infty}(0,T';W^1_p(\Omega))\subset L_{\infty}(0,T';W^1_p(\Omega)).
\end{align*}
Similarly one shows \eqref{H2E:M-mildreg2} and \eqref{H2E:M-mildreg3}.
\end{proof}

\subsection{The main results}\label{H2E:secMain}

We first prove that for $\epsilon>0$ system \eqref{H2E:SystemApu} has a unique M-mild solution and study its convergence as $\epsilon\to 0$. 
\\
\begin{theo}\label{H2E:Maintheorem1}
Assume \eqref{H2E:asssing}. Then
\begin{enumerate}
\item For every $\mu=(h,\epsilon)\in(0,1]\times(0,1], \ 0<T\leq\infty$
system \eqref{H2E:SystemApu} has a unique M-mild solution $\bsym{u}^{\mu}$ defined on $[0,T)$. This solution is nonnegative and is also $\mathcal{X}^{\alpha}$ solution.
\item For every $h\in(0,1], \ \epsilon=0, \ 0<T\leq\infty$ system \eqref{H2E:SystemApu} has a unique M-mild solution $\bsym{u}^{\mu_0}$ defined on $[0,T)$. The solution is nonnegative. Moreover if $T=\infty$ then for every $0<T'<\infty$ the following convergence holds
\begin{align}
\lim_{\epsilon\to0^+}\Big\{\sum_{i=1}^5\n{u_{i}^{\mu}-u_{i}^{\mu_0}}_{L_{\infty}(0,T';Z_i)}\Big\}=0.\label{H2E:limeps}
\end{align}
\end{enumerate} 
\end{theo}

Next we consider the dimension reduction problem. We show that for $\epsilon=0$ the solution of system \eqref{H2E:SystemApu} converges to the solution of an appropriate one dimensional problem when $h\to 0$.\\

\begin{theo}\label{H2E:Maintheorem2}
Let $\bsym{u}^{\mu_0}$ be the unique, global in time M-mild solution of system \eqref{H2E:SystemApu} for $h\in(0,1]$ and $\epsilon=0$. Then for every $0<T<\infty$
\begin{align}
\lim_{h\to0^+}\Big\{\n{t^{2\theta}(z_{1}^{\mu_0}-z_{1}^{0})}_{L_{\infty}(0,T;Z_{1+})}+\sum_{i=2}^5\n{u_{i}^{\mu_0}-u_{i}^{0}}_{L_{\infty}(0,T;Z_i)}\Big\}=0,\label{H2E:limk}
\end{align} 
where $z_{1}^{\mu_0}=u_{1}^{\mu_0}-m^{\mu_0},\ z_1^{0}=E(u_1^{0}-m^{0})$  and $\bsym{u}^{0}=(u_1^{0},\ldots,u_5^{0})$ is the unique classical solution of
\bs\label{H2E:SystemApulim}
\eq{
\partial_t u_{1}-\partial^2_{x_1}u_1+b_1u_1&=f_{1}(\bsym{u})+p_1\delta && (t,x_1)\in I_{\infty}\label{H2E:SystemApuAlim}\\
\partial_t u_{2}-d\partial^2_{x_1} u_{2}&=f_{2}(\bsym{u}),&&(t,x_1)\in I_{\infty}\label{H2E:SystemApuBlim}\\
\partial_t u_{3}&=f_3(\boldsymbol{u}),&&(t,x_1)\in I_{\infty}\label{H2E:SystemApuClim}\\
\partial_t u_{4}&=f_4(\boldsymbol{u}),&&(t,x_1)\in I_{\infty}\label{H2E:SystemApuDlim}\\
\partial_t u_{5}&=f_5(\boldsymbol{u}),&&(t,x_1)\in I_{\infty}\label{H2E:SystemApuElim}
}
\es
with boundary and initial conditions
\begin{align*}
\partial_{x_1}u_1&=\partial_{x_1}u_2=0,&& (t,x_1)\in(\partial I)_T\\
\bsym{u}(0,.)&=\bsym{u}_{0}^0=[Pu_{01},u_{02},u_{03},u_{04},u_{05}].
\end{align*}
\end{theo}

\begin{rem}
Global well-posedness of system \eqref{H2E:SystemApulim} was established in \cite{Mal2}.
\end{rem}

\subsection{Proof of Theorem \ref{H2E:Maintheorem1}}\label{H2E:ProofMain1}

\textbf{Step 1 - estimates for $G_i$'s.}
\\
\
\begin{lem}\label{H2E:Gest}
For $\bsym{z},\bsym{z}'\in \bsym{Z}_+,\ \mu\in\mathbb{Z}_+\times[0,1]$ the following estimates hold
\begin{align*}
&\sum_{i=1}^2\n{G_{i}^{\mu}(\bsym{z})}_{Z_{i-}}+\sum_{i=3}^5\n{G_i(\bsym{z})}_{Z_{i+}}\leq C\Big((1+\sum_{i=1}^2\n{z_i}_{Z_{i+}})(1+\n{z_3}_{Z_{3+}})+\sum_{i=4}^5\n{z_i}_{Z_{i+}}\Big),\\
&\sum_{i=1}^2\n{G_{i}^{\mu}(\bsym{z})-G_{i}^{\mu}(\bsym{z}')}_{Z_{i-}}+\sum_{i=3}^5\n{G_i(\bsym{z})-G_i(\bsym{z}')}_{Z_i}\leq C\Big((1+\n{z_3}_{Z_{3+}}+\n{z_3'}_{Z_{3+}})\sum_{i=1}^2\n{z_i-z_i'}_{Z_i}\\
&+(1+\sum_{i=1}^2(\n{z_i}_{Z_{i+}}+\n{z_i'}_{Z_{i+}}))\n{z_3-z_3'}_{Z_3}+\sum_{i=4}^5\n{z_i-z_i'}_{Z_i}\Big),\\
&\sum_{i=1}^2\n{G_{i}^{\mu}(\bsym{z})-G_{i}^{\mu}(\bsym{z}')}_{Z_{i-}}+\sum_{i=3}^5\n{G_i(\bsym{z})-G_i(\bsym{z}')}_{Z_{i+}}\leq  C\Big((1+\n{z_3}_{Z_{3+}}+\n{z_3'}_{Z_{3+}})\sum_{i=1}^2\n{z_i-z_i'}_{Z_{i+}}\\
&+(1+\sum_{i=1}^2(\n{z_i}_{Z_{i+}}+\n{z_i'}_{Z_{i+}}))\n{z_3-z_3'}_{Z_{3+}}+\sum_{i=4}^5\n{z_i-z_i'}_{Z_{i+}}\Big),\nonumber\\
&\sum_{i=1}^2\n{G_{i}^{\mu}(\bsym{z})-G_{i}^{\mu_0}(\bsym{z})}_{Z_{i-}}\leq C(1+\n{z_3}_{Z_{3+}})\n{\eta^{\epsilon}-\delta}_{X^{-1/4-\theta}(I)},
\end{align*}
where $C$ does not depend on $\mu$.
\end{lem}

\begin{proof}
We will prove inequalities involving $G_{1}^{\mu}$ and $G_3$. Inequalities involving $G_{2}^{\mu},G_4$ and $G_5$ can be derived analogously. Using condition \eqref{H2E:asssing2} and Lemma \ref{H2E:imbeddings} we get that $X^{1/4-\theta}(I)\subset L_p(I)$, from which $Tr\in\mathcal{L}(Z_1,L_p(I))$ by Lemma \ref{H2E:tracelem}. Using above observation and H\"older inequality we estimate
\begin{align*}
&\n{G_{1}^{\mu}(\bsym{z})}_{Z_{1-}}=\n{Tr'g_{1}^{\mu}(Tr(z_1),z_2,z_3,z_4,z_5)}_{X^{-1/4-\theta}(\Omega)}\leq C\n{g_{1}^{\mu}(Tr(z_1),z_2,z_3,z_4,z_5)}_{L_2(I)}\\
&\leq C\Big(\n{Tr(z_1)}_{L_2(I)}+\n{z_2}_{L_{2}(I)}+\n{Tr(z_1)}_{L_2(I)}\n{z_3}_{\infty}+\n{z_4}_{L_2(I)}+\n{z_3}_{L_{2}(I)}\\
&+(1+\n{z_3}_{\infty})\n{Tr(m^{\mu})}_{L_2(I)}\Big)\leq C\Big((1+\sum_{i=1}^2\n{z_i}_{Z_{i+}})(1+\n{z_3}_{Z_{3+}})+\sum_{i=4}^5\n{z_i}_{Z_{i+}}\Big),\\
&\n{G_3(\bsym{z})}_{Z_{3+}}\leq C\Big(\n{z_3}_{\infty}+\n{Tr(z_1)}_{\infty}\n{z_3}_{\infty}+\n{z_2}_{\infty}\n{z_3}_{\infty}+\sum_{i=3}^5\n{z_i}_{\infty}+1\Big)\\
&\leq C\Big((1+\sum_{i=1}^2\n{z_i}_{Z_{i+}})(1+\n{z_3}_{Z_{3+}})+\sum_{i=4}^5\n{z_i}_{Z_{i+}}\Big),\\
&\n{G_{1}^{\mu}(\bsym{z})-G_{1}^{\mu}(\bsym{z}')}_{Z_{1-}}\leq C\n{g_{1}^{\mu}(Tr(z_1),z_2,z_3,z_4,z_5)-g_{1}^{\mu}(Tr(z_1'),z_2',z_3',z_4',z_5')}_{L_2(I)}\\
&\leq C\Big(\n{Tr(z_1-z_1')}_{L_2(I)}+\n{z_2-z_2'}_{L_2(I)}+\n{Tr(z_1-z_1')}_{L_2(I)}\n{z_3}_{\infty}+\n{z_3-z_3'}_{L_{2}(I)}\n{Tr(z_1')}_{\infty}\\
&+\n{z_4-z_4'}_{L_2(I)}+\n{z_3-z_3'}_{L_2(I)}+\n{Tr(m^{\mu})}_{L_{\frac{2p}{p-2}}(I)}\n{z_3-z_3'}_{L_{p}(I)}\Big)\leq C\Big((1+\n{z_3}_{Z_{3+}}\\
&+\n{z_3'}_{Z_{3+}})\sum_{i=1}^2\n{z_i-z_i'}_{Z_i}+(1+\sum_{i=1}^2(\n{z_i}_{Z_{i+}}+\n{z_i'}_{Z_{i+}}))\n{z_3-z_3'}_{Z_3}+\sum_{i=4}^5\n{z_i-z_i'}_{Z_i}\Big),\\
&\n{G_3(\bsym{z})-G_3(\bsym{z}')}_{Z_3}\leq C\Big(\n{z_3-z_3'}_{L_p(I)}+\n{Tr(z_1-z_1')}_{L_p(I)}\n{z_3}_{\infty}+\n{z_3-z_3'}_{L_p(I)}\n{Tr(z_1')}_{\infty}\\
&+\n{z_2-z_2'}_{L_p(I)}\n{z_3}_{\infty}+\n{z_3-z_3'}_{L_p(I)}\n{z_2'}_{\infty}+\sum_{i=4}^5\n{z_i-z_i'}_{L_p(I)}\Big)\leq C\Big((1+\n{z_3}_{Z_{3+}}\\
&+\n{z_3'}_{Z_{3+}})\sum_{i=1}^2\n{z_i-z_i'}_{Z_i}+(1+\sum_{i=1}^2(\n{z_i}_{Z_{i+}}+\n{z_i'}_{Z_{i+}}))\n{z_3-z_3'}_{Z_3}+\sum_{i=4}^5\n{z_i-z_i'}_{Z_i}\Big),\\
&\n{G_3(\bsym{z})-G_3(\bsym{z}')}_{Z_{3+}}\leq C\Big(\n{z_3-z_3'}_{\infty}+\n{Tr(z_1-z_1')}_{\infty}\n{z_3}_{\infty}+\n{z_3-z_3'}_{\infty}\n{Tr(z_1')}_{\infty}\\
&+\n{z_2-z_2'}_{\infty}\n{z_3}_{\infty}+\n{z_3-z_3'}_{\infty}\n{z_2'}_{\infty}+\sum_{i=4}^5\n{z_i-z_i'}_{\infty}\Big)\leq C\Big((1+\n{z_3}_{Z_{3+}}\\
&+\n{z_3'}_{Z_{3+}})\sum_{i=1}^2\n{z_i-z_i'}_{Z_{i+}}+(1+\sum_{i=1}^2(\n{z_i}_{Z_{i+}}+\n{z_i'}_{Z_{i+}}))\n{z_3-z_3'}_{Z_{3+}}+\sum_{i=4}^5\n{z_i-z_i'}_{Z_{i+}}\Big),\\
&\n{G_{1}^{\mu}(\bsym{z})-G_{1}^{\mu_0}(\bsym{z})}_{Z_{1-}}\leq C\n{(c_1+z_3)Tr(m^{\mu}-m^{\mu_0})}_{L_2(I)}\leq C(1+\n{z_3}_{\infty})\n{Tr(m^{\mu}-m^{\mu_0})}_{L_{2}(I)}\\
&\leq C(1+\n{z_3}_{Z_{3+}})\n{\eta^{\epsilon}-\delta}_{X^{-1/4-\theta}(I)}.
\end{align*}
\end{proof}

\textbf{Step 2 - uniqueness of M-mild solution to system \eqref{H2E:SystemApu}.}
\

Assume that $\bsym{u},\bsym{u}'$ are two M-mild solutions of system \eqref{H2E:SystemApu} on $[0,T), 0<T\leq\infty$, with the same initial condition. Let $\bsym{z},\bsym{z}'$ be related with  $\bsym{u},\bsym{u}'$ by \eqref{H2E:zmitrans}, \eqref{H2E:zmi0trans}. Fix $T'<T$.

For $t\in(0,T')$ denote $f(t)=\sum_{i=1}^5\n{z_{i}(t)-z_{i}'(t)}_{Z_i}$. 
Put
\begin{align*}
&K_1(T')=\n{t^{2\theta}z_{1}}_{L_{\infty}(0,T';Z_{1+})}+\n{t^{2\theta}z_{1}'}_{L_{\infty}(0,T';Z_{1+})}\\,
&K_{i}(T')=\n{z_{i}}_{L_{\infty}(0,T';Z_{i+})}+\n{z_{i}'}_{L_{\infty}(0,T';Z_{i+})}, \ i=2,3\\
&\overline{K}(T')=\max\{K_1(T'),K_2(T'),K_3(T')\}
\end{align*}
Using condition \eqref{H2E:deficond} we get that $f\in L_{\infty}(0,T')$ and $\overline{K}(T')<\infty$. Owing to Lemma \ref{H2E:Gest} we obtain that for $t\in(0,T')$ there is
\begin{align*}
&\sum_{i=1}^2\n{G_{i}^{\mu}(\bsym{z}(t))-G_{i}^{\mu}(\bsym{z}'(t))}_{Z_{i-}}+\sum_{i=3}^5\n{G_{i}(\bsym{z}(t))-G_{i}(\bsym{z}'(t))}_{Z_i}\leq C\Big\{\Big(1+\n{z_3(t)}_{Z_{3+}}\\
&+\n{z_3'(t)}_{Z_{3+}}\Big)\sum_{i=1}^2\n{z_{i}(t)-z'_{i}(t)}_{Z_i}+\Big(1+\sum_{i=1}^2(\n{z_i(t)}_{Z_{i+}}+\n{z_i'(t)}_{Z_{i+}})\Big)\n{z_3(t)-z_3'(t)}_{Z_3}\\
&+\sum_{i=4}^5\n{z_i(t)-z_i'(t)}_{Z_i}\Big\}\leq C\Big\{\Big(1+K_3(T')\Big)\sum_{i=1}^2\n{z_{i}(t)-z'_{i}(t)}_{Z_i}+\Big(1+\frac{1}{t^{2\theta}}K_1(T')\\
&+K_2(T')\Big)\n{z_3(t)-z_3'(t)}_{Z_3}+\sum_{i=4}^5\n{z_i(t)-z_i'(t)}_{Z_i}\Big\}\leq C(1+\overline{K}(T'))\Big(1+\frac{1}{t^{2\theta}}\Big)f(t).
\end{align*}
Using Lemma \ref{H2E:sembasicest} and owing to the fact that $\bsym{z},\bsym{z}'$ satisfy \eqref{H2E:Duhamelz} we obtain for $t\in(0,T')$
\begin{align*}
&f(t)\leq\int_0^t\Big\{\n{e^{(t-\tau)A_h}}_{\mathcal{L}(Z_{1-},Z_1)}\n{G_{1}^{\mu}(\bsym{z}(\tau))-G_{1}^{\mu}(\bsym{z}'(\tau))}_{Z_{1-}}\\
&+\n{e^{(t-\tau)dA_{0}}}_{\mathcal{L}(Z_{2-},Z_2)}\n{G_{2}^{\mu}(\bsym{z}(\tau))-G_{2}^{\mu}(\bsym{z}'(\tau))}_{Z_{2-}}+\n{e^{-(t-\tau)Tr(m^{\mu})}}_{Z_{3+}}\n{G_3(\bsym{z}(\tau))-G_3(\bsym{z}'(\tau))}_{Z_3}\\
&+\sum_{i=4}^5\n{G_i(\bsym{z}(\tau))-G_i(\bsym{z}'(\tau))}_{Z_i}\Big\}d\tau\leq C(1+\overline{K}(T'))\int_0^t\Big(1+\frac{1}{(t-\tau)^{3/4}}+\frac{1}{(t-\tau)^{1/2}}\Big)\Big(1+\frac{1}{\tau^{2\theta}}\Big)f(\tau)d\tau\\
&\leq C(1+\overline{K}(T'))(1+(T')^{3/4+2\theta})\int_0^t\frac{f(\tau)}{\tau^{2\theta}(t-\tau)^{3/4}}d\tau.
\end{align*}
Finally using  Lemma \eqref{H2E:Gronlem} (see \eqref{H2E:asssing2}) we conclude that $f\equiv 0$ on $(0,T')$ hence $\bsym{u}\equiv\bsym{u}'$.\\

\textbf{Step 3 - existence of global solutions for $\epsilon>0$ and $\mu$-independence of bounds.}

Using Theorem \ref{H2E:existencereg} with $s=1/2+\theta,\ s'=1/2+2\theta, \ \omega=\eta^{\epsilon}$ we obtain that system \eqref{H2E:SystemApu} has for $\mu\in(0,1]\times(0,1]$ a unique global $\mathcal{X}^\alpha$ solution $\bsym{u}^{\mu}$ which is nonnegative. Let $\bsym{z}^{\mu}, \bsym{z}_{0}^{\mu}$ be related with $\bsym{u}^{\mu}, \bsym{u}_{0}$ by \eqref{H2E:zmitrans} and \eqref{H2E:zmi0trans}. It is easy to see that $\bsym{z}^{\mu}$ satisfies formulas \eqref{H2E:Duhamelz} from which one concludes that  $\bsym{u}^{\mu}$ is also a M-mild solution of system \eqref{H2E:SystemApu}. Using estimate \eqref{H2E:linftybound} from Theorem \ref{H2E:existencereg} we get that
\begin{align}
M_3=\sup_{\mu\in(0,1]\times(0,1]}\sum_{i=3}^5\n{z_{i}^{\mu}}_{L_{\infty}(0,\infty;Z_{i+})}\label{H2E:M3}
\end{align}
is finite. Fix $T<\infty$ and for $0<t<T$ denote $g(t)=1+t^{2\theta}\n{z_{1}^{\mu}(t)}_{Z_{1+}}+\n{z_{2}^{\mu}(t)}_{Z_{2+}}$. Owing to Lemma \ref{H2E:Gest} we obtain
\begin{align*}
&\sum_{i=1}^2\n{G_{i}^{\mu}(\bsym{z}^{\mu}(t))}_{Z_{i-}}\leq C\Big((1+\n{z_{1}^{\mu}(t)}_{Z_{1+}}+\n{z_{2}^{\mu}(t)}_{Z_{2+}})(1+\n{z_{3}^{\mu}(t)}_{Z_{3+}})+\n{z_{4}^{\mu}(t)}_{Z_{4+}}+\n{z_{5}^{\mu}(t)}_{Z_{5+}}\Big)\\
&\leq C(1+M_3)(1+\n{z_{1}^{\mu}(t)}_{Z_{1+}}+\n{z_{2}^{\mu}(t)}_{Z_{2+}})\leq C(1+M_3)\Big(1+\frac{1}{t^{2\theta}}\Big)g(t).
\end{align*}

Using \eqref{H2E:Duhamelz} and Lemma \ref{H2E:sembasicest} we estimate (recall that $Z_2=Z_{2+}$)
\begin{align*}
g(t)&\leq 1+t^{2\theta}\n{e^{tA_h}}_{\mathcal{L}(Z_1,Z_{1+})}\n{z_{01}^{\mu}}_{Z_1}+\n{e^{tdA_{0}}}_{\mathcal{L}(Z_2)}\n{z_{02}}_{Z_2}\\
&+\int_0^t\Big(t^{2\theta}\n{e^{(t-\tau)A_h}}_{\mathcal{L}(Z_{1-},Z_{1+})}\n{G_{1}^{\mu}(\bsym{z}^{\mu}(\tau))}_{Z_{1-}}+\n{e^{(t-\tau)dA_{0}}}_{\mathcal{L}(Z_{2-},Z_2)}\n{G_{2}^{\mu}(\bsym{z}^{\mu}(\tau))}_{Z_{2-}}\Big)d\tau\\
&\leq C(1+t^{2\theta})\Big\{1+(1+M_3)\int_0^t\Big(1+\frac{1}{(t-\tau)^{3/4+2\theta}}+\frac{1}{(t-\tau)^{1/2}}\Big)\Big(1+\frac{1}{\tau^{2\theta}}\Big)g(\tau)d\tau\Big\}\\
&\leq C(1+T^{2\theta})\Big\{1+(1+M_3)(1+T^{3/4+4\theta})\int_0^t\frac{g(\tau)}{(t-\tau)^{3/4+2\theta}\tau^{2\theta}}d\tau\Big\}\\
&\leq C(1+T^{2\theta})+C(1+M_3)(1+T^{3/4+6\theta})\int_0^t\frac{g(\tau)}{(t-\tau)^{3/4+2\theta}\tau^{2\theta}}d\tau.
\end{align*}

Thus using Lemma \ref{H2E:Gronlem} (see \eqref{H2E:asssing2}) we get that for every $T>0$
\begin{align}
M_1(T)=\sup_{\mu\in(0,1]\times(0,1]}\n{t^{2\theta}z_{1}^{\mu}}_{L_{\infty}(0,T;Z_{1+})} \ {\rm and} \ M_2(T)=\sup_{\mu\in(0,1]\times(0,1]}\n{z_{2}^{\mu}}_{L_{\infty}(0,T;Z_{2+})}\label{H2E:M12}
\end{align}
are finite.

\textbf{Step 4 - existence of local M-mild solutions for $\epsilon=0$.}

To prove existence of local M-mild solutions we use the contraction mapping principle in appropriate weighted in time spaces. For $R, T>0$ define  
\begin{align*}
&\mathcal{Z}_1=\{z_1\in C([0,T],Z_1):\ \n{z_1}_{L_{\infty}(0,T;Z_1)}+\n{t^{2\theta}z_1}_{L_{\infty}(0,T;Z_{1+})}\leq R \},\\ &d_{\mathcal{Z}_1}(z_1,z_1')=\n{z_1-z_1'}_{L_{\infty}(0,T;Z_1)}+\n{t^{2\theta}(z_1-z_1')}_{L_{\infty}(0,T;Z_{1+})},\\
&\mathcal{Z}_2=\{z_2\in C([0,T],Z_2): \n{z_2}_{L_{\infty}(0,T;Z_2)}\leq R\}, \ d_{\mathcal{Z}_2}(z_2,z_2')=\n{z_2-z_2'}_{L_{\infty}(0,T;Z_2)},\\
&\mathcal{Z}_i=\{z_i\in C([0,T],Z_i): \n{z_i}_{L_{\infty}(0,T;Z_{i+})}\leq R\}, \ d_{\mathcal{Z}_i}(z_i,z_i')=\n{z_i-z_i'}_{L_{\infty}(0,T;Z_{i+})}, \ i=3,4,5\\
&\bsym{\mathcal{Z}}=\mathcal{Z}_1\times\ldots\times\mathcal{Z}_5, \ d_{\bsym{\mathcal{Z}}}(\bsym{z},\bsym{z}')=\sum_{i=1}^5d_{\mathcal{Z}_i}(z_i,z_i').
\end{align*}
Observe that $\mathcal{Z}_i$ and $\bsym{\mathcal{Z}}$ are complete metric spaces.

For $\bsym{z}\in\bsym{\mathcal{Z}}, \ \mu=(h,\epsilon)\in(0,1]\times[0,1]$ define
\begin{align*}
[\Phi_{1}^{\mu}(\bsym{z})](t)&=e^{t(A_h-b_1)}z_{01}^{\mu}+\int_0^te^{(t-\tau)(A_h-b_1)}G_{1}^{\mu}(\bsym{z}(\tau))d\tau,\\
[\Phi_{2}^{\mu}(\bsym{z})](t)&=e^{tdA_0}z_{02}+\int_0^te^{(t-\tau)dA_{0}}G_{2}^{\mu}(\bsym{z}(\tau))d\tau,\\
[\Phi_{3}^{\mu}(\bsym{z})](t)&=e^{-tTr(m^{\mu})}z_{03}+\int_0^te^{-(t-\tau)Tr(m^{\mu})}G_{3}(\bsym{z}(\tau))d\tau,\\
[\Phi_{i}(\bsym{z})](t)&=z_{0i}+\int_0^tG_i(\bsym{z}(\tau))d\tau, \ i=4,5\\
\bsym{\Phi}^{\mu}&=(\Phi_{1}^{\mu},\Phi_{2}^{\mu},\Phi_{3}^{\mu},\Phi_{4},\Phi_{5}).
\end{align*}

\begin{lem}\label{H2E:lemunifball}
There exist $R, T>0$ such that for every $\mu\in(0,1]\times[0,1]$ the map $\bsym{\Phi}^{\mu}$ maps $\bsym{\mathcal{Z}}$ into itself and satisfies for every $\bsym{z},\bsym{z}'\in\bsym{\mathcal{Z}}$ the following condition
\begin{align}
d_{\bsym{\mathcal{Z}}}(\bsym{\Phi}^{\mu}(\bsym{z}),\bsym{\Phi}^{\mu}(\bsym{z}'))\leq (1/2)d_{\bsym{\mathcal{Z}}}(\bsym{z},\bsym{z}')\label{H2E:contract}.
\end{align}
\end{lem}

\begin{proof}
\
Fix $R\geq 1\geq T>0$. Using Lemma \ref{H2E:Gest}  we have for $t\in[0,T]$ and  $\bsym{z},\bsym{z}'\in\bsym{\mathcal{Z}}$
\begin{align}
&\sum_{i=1}^2\n{G_{i}^{\mu}(\bsym{z}(t))}_{Z_{i-}}+\sum_{i=3}^5\n{G_i(\bsym{z}(t))}_{Z_{i+}}\leq CR^2\Big(1+\frac{1}{t^{2\theta}}\Big)\label{H2E:blah1}\\
&\sum_{i=1}^2\n{G_{i}^{\mu}(\bsym{z}(t))-G_{i}^{\mu}(\bsym{z}'(t))}_{Z_{i-}}+\sum_{i=3}^5\n{G_i(\bsym{z}(t))-G_i(\bsym{z}'(t))}_{Z_{i+}}\leq CR\Big(1+\frac{1}{t^{2\theta}}\Big)d_{\bsym{\mathcal{Z}}}(\bsym{z},\bsym{z}')\label{H2E:blah2}. 
\end{align}
Using \eqref{H2E:blah1} and Lemma \ref{H2E:sembasicest} we estimate 
\begin{align*}
&t^{2\theta}\n{[\Phi_{1}^{\mu}(\bsym{z})](t)}_{Z_{1+}}+\n{[\Phi_{1}^{\mu}(\bsym{z})](t)}_{Z_1}+\sum_{i=2}^5\n{[\Phi_{i}^{\mu}(\bsym{z})](t)}_{Z_{i+}}\leq (t^{2\theta}\n{e^{t A_h}}_{\mathcal{L}(Z_1,Z_{1+})}+\n{e^{t A_h}}_{\mathcal{L}(Z_1)})\n{z_{01}^{\mu}}_{Z_1}\\
&+\n{e^{tdA_0}}_{\mathcal{L}(Z_2)}\n{z_{02}}_{Z_2}+\n{e^{-tTr(m^{\mu})}}_{Z_{3+}}\n{z_{03}}_{Z_{3+}}+\sum_{i=4}^5\n{z_{0i}}_{Z_{i+}}
+\int_0^t\Big\{\Big(t^{2\theta}\n{e^{(t-\tau) A_h}}_{\mathcal{L}(Z_{1-},Z_{1+})}\\
&+\n{e^{(t-\tau) A_h}}_{\mathcal{L}(Z_{1-},Z_1)}\Big)\n{G_{1}^{\mu}(\bsym{z}(\tau))}_{Z_{1-}}+\n{e^{(t-\tau)dA_0}}_{\mathcal{L}(Z_{2-},Z_2)}\n{G_{2}^{\mu}(\bsym{z}(\tau))}_{Z_{2-}}\\
&+\n{e^{-(t-\tau)Tr(m^{\mu})}}_{Z_{3+}}\n{G_3(\bsym{z}(\tau))}_{Z_{3+}}+\sum_{i=4}^5\n{G_i(\bsym{z}(\tau))}_{Z_{i+}}\Big\}d\tau\leq C(t^{2\theta}+1)\Big\{\n{z_{01}^{\mu}}_{Z_1}+\n{z_{02}}_{Z_2}\\
&+\sum_{i=3}^5\n{z_{0i}}_{Z_{i+}}+R^2\int_0^t\Big(1+\frac{1}{(t-\tau)^{3/4+2\theta}}+\frac{1}{(t-\tau)^{3/4}}+\frac{1}{(t-\tau)^{1/2}}\Big)\Big(1+\frac{1}{\tau^{2\theta}}\Big)d\tau\Big\}\\
&\leq C\Big\{\n{z_{01}^{\mu}}_{Z_1}+\n{z_{02}}_{Z_2}+\sum_{i=3}^5\n{z_{0i}}_{Z_{i+}}+R^2\int_0^t\frac{1}{(t-\tau)^{3/4+2\theta}\tau^{2\theta}}d\tau\Big\}\leq C(\n{z_{01}^{\mu}}_{Z_1}+\n{z_{02}}_{Z_2}\\
&+\sum_{i=3}^5\n{z_{0i}}_{Z_{i+}})+CR^2T^{1/4-4\theta}.
\end{align*}
Taking $R,T$ such that $R\geq \max\{1,2C(\n{z_{01}^{\mu}}_{Z_1}+\n{z_{02}}_{Z_2}+\sum_{i=3}^5\n{z_{0i}}_{Z_{i+}})\}$ and $T\leq\min\{1,(2CR)^{4/(16\theta-1)}\}$ we obtain 
\begin{align*}
t^{2\theta}\n{[\Phi_{1}^{\mu}(\bsym{z})](t)}_{Z_{1+}}+\n{[\Phi_{1}^{\mu}(\bsym{z})](t)}_{Z_1}+\sum_{i=2}^5\n{[\Phi_{i}^{\mu}(\bsym{z})](t)}_{Z_{i+}}\leq R/2+R/2=R
\end{align*}
which proves that $\bsym{\Phi}^{\mu}$ maps $\bsym{\mathcal{Z}}$ into itself. Using \eqref{H2E:blah2} we prove analogously that condition \eqref{H2E:contract} holds after making $T$ smaller if needed.
\end{proof}

We obtain from Lemma \ref{H2E:lemunifball} that the map $\bsym{\Phi}^{\mu}:\bsym{\mathcal{Z}}\to\bsym{\mathcal{Z}}$ satisfies, for certain $R,T$ which are independent of $\mu$, the assumptions of the contraction mapping principle.  We conclude that system \eqref{H2E:SystemApu} has for $\epsilon=0$ a unique maximally defined M-mild solution $\bsym{u}^{\mu_0}$ defined on $[0,T_{\max}^h)$, where $T^*:=\inf\{T_{\max}^{h}: h\in(0,1]\}>0$.

\textbf{Step 5 - For any fixed $h\in(0,1]$:  $\bsym{u}^{\mu}$ converges to $\bsym{u}^{\mu_0}$ as $\epsilon\to 0$. Moreover $T_{\max}^h=\infty$.}

Fix $T<T_{\max}^h$ and for $0<t<T$ denote: $f^{\mu}(t)=\sum_{i=1}^5\n{z_{i}^{\mu}(t)-z_{i}^{\mu_0}(t)}_{Z_i}$. Put
\begin{align*}
K_{1}^h(T)&=\sup_{\epsilon\in[0,1]}\n{t^{2\theta}z_{1}^{\mu}}_{L_{\infty}(0,T;Z_{1+})}, \ K_{i}^h(T)=\sup_{\epsilon\in[0,1]}\n{z_{i}^{\mu}}_{L_{\infty}(0,T;Z_{i+})}, \ i=2,3
\end{align*}
Observe that due to \eqref{H2E:M3},\eqref{H2E:M12} $K_{i}^h(T)$ are finite. Denote $\overline{K}^h(T)=\max\{K_{1}^h(T),K_{2}^h(T),K_{3}^h(T)\}$.
Using Lemma \ref{H2E:Gest}  we have for $0<t<T$
\begin{align*}
&\sum_{i=1}^2\n{G_{i}^{\mu}(\bsym{z}^{\mu}(t))-G_{i}^{\mu}(\bsym{z}^{\mu_0}(t))}_{Z_{i-}}+\sum_{i=3}^5\n{G_i(\bsym{z}^{\mu}(t))-G_i(\bsym{z}^{\mu_0}(t))}_{Z_i}\leq C(1+\overline{K}^h(T))\Big(1+\frac{1}{t^{2\theta}}\Big)f^{\mu}(t)\\
&\sum_{i=1}^2\n{G_{i}^{\mu}(\bsym{z}^{\mu_0}(t))-G_{i}^{\mu_0}(\bsym{z}^{\mu_0}(t))}_{Z_{i-}}\leq C(1+\overline{K}^h(T))\n{\eta^{\epsilon}-\delta}_{X^{-1/4-\theta}(I)}\\
&\n{G_3(\bsym{z}^{\mu_0}(t))}_{Z_{3+}}\leq C\Big(1+\frac{1}{t^{2\theta}}\Big)\Big(1+(\overline{K}^h(T))^2\Big)
\end{align*}
Thus owing to \eqref{H2E:Duhamelz} we estimate
\begin{align*}
&f^{\mu}(t)\leq\n{e^{tA_h}}_{\mathcal{L}(Z_1)}\n{z_{01}^{\mu}-z_{01}^{\mu_0}}_{Z_1}+\n{e^{-tTr(m^{\mu})}-e^{-tTr(m^{\mu_0})}}_{Z_3}\n{z_{03}}_{Z_{3+}}\\
&+\int_0^t\Big\{\n{e^{(t-\tau)A_h}}_{\mathcal{L}(Z_{1-},Z_1)}\Big(\n{G_{1}^{\mu}(\bsym{z}^{\mu}(\tau))-G_{1}^{\mu}(\bsym{z}^{\mu_0}(\tau))}_{Z_{1-}}+\n{G_{1}^{\mu}(\bsym{z}^{\mu_0}(\tau))-G_{1}^{\mu_0}(\bsym{z}^{\mu_0}(\tau))}_{Z_{1-}}\Big)\\
&+\n{e^{(t-\tau)dA_{0}}}_{\mathcal{L}(Z_{2-},Z_2)}\Big(\n{G_{2}^{\mu}(\bsym{z}^{\mu}(\tau))-G_{2}^{\mu}(\bsym{z}^{\mu_0}(\tau))}_{Z_{2-}}+\n{G_{2}^{\mu}(\bsym{z}^{\mu_0}(\tau))-G_{2}^{\mu_0}(\bsym{z}^{\mu_0}(\tau))}_{Z_{2-}}\Big)\\
&+\Big(\n{e^{-(t-\tau)Tr(m^{\mu})}}_{Z_{3+}}\n{G_3(\bsym{z}^{\mu}(\tau))-G_3(\bsym{z}^{\mu_0}(\tau))}_{Z_3}\\
&+\n{e^{-(t-\tau)Tr(m^{\mu})}-e^{-(t-\tau)Tr(m^{\mu_0})}}_{Z_3}\n{G_3(\bsym{z}^{\mu_0}(\tau))}_{Z_{3+}}\Big)+\sum_{i=4}^5\n{G_i(\bsym{z}^{\mu}(\tau))-G_i(\bsym{z}^{\mu_0}(\tau))}_{Z_i}\Big\}d\tau.
\end{align*}

Using Lemma \ref{H2E:sembasicest} and Lemma \eqref{H2E:swallow} we obtain
\begin{align*}
&f^{\mu}(t)\leq C\n{m^{\mu}-m^{\mu_0}}_{Z_1}+t\n{Tr(m^{\mu}-m^{\mu_0})}_{Z_3}\n{z_{03}}_{Z_{3+}}\\
&+\int_0^t\Big\{\Big(\n{e^{(t-\tau)A_h}}_{\mathcal{L}(Z_{1-},Z_1)}+\n{e^{(t-\tau)dA_{0}}}_{\mathcal{L}(Z_{2-},Z_2)}+\n{e^{-(t-\tau)Tr(m^{\mu})}}_{Z_{3+}}+1\Big)\Big(\sum_{i=1}^2\n{G_{i}^{\mu}(\bsym{z}^{\mu}(\tau))\\
&-G_{i}^{\mu}(\bsym{z}^{\mu_0}(\tau))}_{Z_{i-}}+\sum_{i=3}^5\n{G_i(\bsym{z}^{\mu}(\tau))-G_i(\bsym{z}^{\mu_0}(\tau))}_{Z_i}\Big)\Big\}d\tau\\
&+\int_0^t\Big\{\Big(\n{e^{(t-\tau)(A_h-b_1)}}_{\mathcal{L}(Z_{1-},Z_1)}+\n{e^{(t-\tau)dA_{0}}}_{\mathcal{L}(Z_{2-},Z_2)}\Big)\Big(\sum_{i=1}^2\n{G_{i}^{\mu}(\bsym{z}^{\mu_0}(\tau))-G_{i}^{\mu_0}(\bsym{z}^{\mu_0}(\tau))}_{Z_{i-}}\Big)\Big\}d\tau\\
&+\int_0^t\n{e^{-(t-\tau)Tr(m^{\mu})}-e^{-(t-\tau)Tr(m^{\mu_0})}}_{Z_3}\n{G_3(\bsym{z}^{\mu_0}(\tau))}_{Z_{3+}}d\tau\\
&\leq C(1+t)\n{\eta^{\epsilon}-\delta}_{X^{-1/4-\theta}(I)}+C(1+\overline{K}^h(T))\int_0^t\Big(1+\frac{1}{(t-\tau)^{3/4}}+\frac{1}{(t-\tau)^{1/2}}\Big)\Big(1+\frac{1}{\tau^{2\theta}}\Big)f^{\mu}(\tau)d\tau\\
&+C(1+\overline{K}^h(T))\n{\eta^{\epsilon}-\delta}_{X^{-1/4-\theta}(I)}\int_0^t\Big(\frac{1}{(t-\tau)^{3/4}}+\frac{1}{(t-\tau)^{1/2}}\Big)d\tau\\
&+C\Big(1+(\overline{K}^h(T))^2\Big)\n{\eta^{\epsilon}-\delta}_{X^{-1/4-\theta}(I)}\int_0^t(t-\tau)\Big(1+\frac{1}{\tau^{2\theta}}\Big)d\tau\\
&\leq a_h(T)\n{\eta^{\epsilon}-\delta}_{X^{-1/4-\theta}(I)}+b_h(T)\int_0^t\frac{f^{\mu}(\tau)}{(t-\tau)^{3/4}\tau^{2\theta}}d\tau.
\end{align*}
Using Lemma \ref{H2E:Gronlem}  (see \eqref{H2E:asssing2}) we get that
\begin{align*}
\n{f^{\mu}}_{L_{\infty}(0,T)}\leq a_h(T)\n{\eta^{\epsilon}-\delta}_{X^{-1/4-\theta}(I)}C\exp\Big(b_h(T)^{\frac{7/4+2\theta}{1/4-2\theta}}CT^{7/4+2\theta}\Big),
\end{align*}
from which we conclude that $\lim_{\epsilon\to0^+}\n{f^{\mu}}_{L_{\infty}(0,T)}=0$ for every $h\in(0,1], \ T<T_{max}^h$by Lemma \ref{H2E:deltacon}. In particular 
$\bsym{u}^{\mu_0}$ is nonnegative on $[0,T_{max}^h)$ and for every $T<T_{max}^h$
\begin{align*}
&\n{t^{2\theta}z_{1}^{\mu_0}}_{L_{\infty}(0,T;Z_{1+})}\leq M_1(T),\\
&\n{z_{2}^{\mu_0}}_{L_{\infty}(0,T;Z_{2+})}\leq M_2(T),\\
&\sum_{i=3}^5\n{z_{i}^{\mu_0}}_{L_{\infty}(0,T;Z_{i+})}\leq M_3,
\end{align*}
where $M_1(T),M_2(T)$ are defined in \eqref{H2E:M12} while $M_3$ is defined in \eqref{H2E:M3}.
We observe thus that $\bsym{z}^{\mu_0}$ does not blow-up in finite time in $\n{\cdot}_{\bsym{Z}_+}$. Using standard continuation argument we conclude that $T_{max}^h=\infty$ for any $h\in(0,1]$.

\subsection{Proof of Theorem \ref{H2E:Maintheorem2}}\label{H2E:ProofMain2}

Recall that $m^0$ is defined in \eqref{H2E:auxdef} while $u_{01}$ in \eqref{H2E:asssing3}. Denote 
\begin{align*}
z_{01}^0&=Pu_{01}-m^0,\\
g_{1}^0,g_{2}^0&: I\times\mathbb{R}^5\to\mathbb{R},\\
g_{1}^0(\boldsymbol{z})&=-c_1z_1+c_2z_2-z_1z_3+c_4(z_4-z_3)-(c_1+z_3)m^0,\\
g_{2}^0(\boldsymbol{z})&=-b_2z_2+c_1z_1-c_2z_2-c_3z_2z_3+c_5(z_5-z_4)+c_1m^0.
\end{align*}
For $\bsym{z}\in\bsym{Z}_+$ define
\begin{align*}
G_{1}^0(\bsym{z})&=Tr'(g_{1}^0(Tr(z_1),z_2,z_3,z_4,z_5)),\\
G_{2}^0(\bsym{z})&=g_{2}^0(Tr(z_1),z_2,z_3,z_4,z_5).
\end{align*}

Observe that since $\bsym{u}^0=(u_1^0,\ldots,u_5^0)$ solves \eqref{H2E:SystemApuAlim}, $\bsym{z}^0=(z_{1}^0,\ldots,z_{5}^0)=M(E(u_{1}^0-m^0),u_{2}^0,\ldots,u_{5}^0)$ satisfies the following Duhamel formulas:
\bs\label{H2E:Duhamelzinfty}
\eq{
z_{1}^0(t)&=E\Big\{e^{t(A_{0}-b_1)}z_{01}^0+\int_0^te^{(t-\tau)(A_0-b_1)}PG_{1}^0(\bsym{z}^0(\tau))d\tau\Big\},\\
z_{2}^0(t)&=e^{tdA_{0}}z_{02}+\int_0^te^{(t-\tau)dA_{0}}G_{2}^0(\bsym{z}^0(\tau))d\tau,\\
z_{3}^0(t)&=e^{-tm^0}z_{03}+\int_0^te^{-(t-\tau)m^0}G_3(\bsym{z}^0(\tau))d\tau,\\
z_{i}^0(t)&=z_{0i}+\int_0^tG_i(\bsym{z}^0(\tau))d\tau, \ i\in\{4,5\}.
}
\es
For $t<T<\infty$ denote 
\begin{align*}
N(T)&=\sup_{h\in(0,1]}\Big(\n{t^{2\theta}z_{1}^{\mu_0}}_{L_{\infty}(0,T;Z_{1+})}+\n{t^{2\theta}z_{1}^0}_{L_{\infty}(0,T;Z_{1+})}+\sum_{i=2}^3(\n{z_{i}^{\mu_0}}_{L_{\infty}(0,T;Z_{i+})}+\n{z_{i}^0}_{L_{\infty}(0,T;Z_{i+})})\Big),\\
f^{\mu_0}(t)&=t^{2\theta}\n{z_{1}^{\mu_0}(t)-z_{1}^{0}(t)}_{Z_{1+}}+\sum_{i=2}^5\n{z_{i}^{\mu_0}(t)-z_{i}^{0}(t)}_{Z_i}.
\end{align*}
Observe that $N(T)\leq M_1(T)+M_2(T)+M_3<\infty$ as was proved in Step 3 of Theorem \ref{H2E:Maintheorem1}.
Owing to Lemma \ref{H2E:Gest} and Lemma \ref{H2E:swallow} we have
\begin{align*}
&\sum_{i=1}^2\n{G_{i}^{\mu_0}(\bsym{z}^{\mu_0}(t))-G_{i}^{\mu_0}(\bsym{z}^{0}(t))}_{Z_{i-}}+\sum_{i=3}^5\n{G_i(\bsym{z}^{\mu_0}(t))-G_i(\bsym{z}^{0}(t))}_{Z_i}\leq C(1+N(T))\Big(1+\frac{1}{t^{2\theta}}\Big)f^{\mu_0}(t)\\
&\n{G_{1}^{0}(\bsym{z}^{0}(t))}_{Z_{1-}}+\n{G_3(\bsym{z}^{0}(t))}_{Z_{3+}}\leq C(1+(N(T))^2)\Big(1+\frac{1}{t^{2\theta}}\Big)\\
&\sum_{i=1}^2\n{G_{i}^{\mu_0}(\bsym{z}^{0}(t))-G_{i}^{0}(\bsym{z}^{0}(t))}_{Z_{i-}}\leq C(1+N(T))\frac{1}{|\lambda_{01,h}^{\Omega}|^{\theta/2}}.
\end{align*}

Since $\bsym{z}^{\mu_0}$ (resp. $\bsym{z}^{0}$) satisfies \eqref{H2E:Duhamelz} (resp. \eqref{H2E:Duhamelzinfty}) thus using \eqref{H2E:Identities} and Lemma \ref{H2E:Multi} we obtain

\begin{align*}
&f^{\mu_0}(t)\leq t^{2\theta}\n{e^{tA_h}(z_{01}^{\mu_0}-Ez_{01}^0)}_{Z_{1+}}+\n{e^{-tTr(m^{\mu_0})}-e^{-tm^0}}_{Z_3}\n{z_{03}}_{Z_{3+}}\\
&+\int_0^t\Big\{t^{2\theta}\n{e^{(t-\tau)A_h}\Big(G_{1}^{\mu_0}(\bsym{z}^{\mu_0}(\tau))-EPG_{1}^0(\bsym{z}^0(\tau))\Big)}_{Z_{1+}}\Big\}d\tau+\int_0^t\Big\{\n{e^{(t-\tau)dA_{0}}\Big(G_{2}^{\mu_0}(\bsym{z}^{\mu_0}(\tau))\\
&-G_{2}^{0}(\bsym{z}^0(\tau))\Big)}_{Z_{2}}\Big\}d\tau+\int_0^t\Big\{\n{e^{-(t-\tau)Tr(m^{\mu_0})}G_3(\bsym{z}^{\mu_0}(\tau))-e^{-(t-\tau)m^0}G_3(\bsym{z}^0(\tau))}_{Z_3}\Big\}d\tau\\
&+\sum_{i=4}^5\int_0^t\n{G_i(\bsym{z}^{\mu_0}(\tau))-G_i(\bsym{z}^0(\tau))}_{Z_i}d\tau\leq t^{2\theta}\n{e^{tA_h}(I-EP)u_{01}}_{Z_{1+}}+t^{2\theta}\n{e^{tA_h}}_{\mathcal{L}(Z_1,Z_{1+})}\n{m^{\mu_0}-Em^0}_{Z_1}\\
&+t\n{Tr(m^{\mu_0})-m^0}_{Z_3}\n{z_{03}}_{Z_{3+}}+T^{2\theta}\int_0^t\Big\{\n{e^{(t-\tau)A_h}}_{\mathcal{L}(Z_{1-},Z_{1+})}\Big(\n{G_{1}^{\mu_0}(\bsym{z}^{\mu_0}(\tau))-G_{1}^{\mu_0}(\bsym{z}^0(\tau))}_{Z_{1-}}\\
&+\n{G_{1}^{\mu_0}(\bsym{z}^0(\tau))-G_{1}^0(\bsym{z}^0(\tau))}_{Z_{1-}}\Big)\Big\}d\tau+T^{2\theta}\int_0^t\Big\{\n{e^{(t-\tau)A_h}(I-EP)G_{1}^0(\bsym{z}^0(\tau))}_{Z_{1+}}\Big\}d\tau\\
&+\int_0^t\Big\{\n{e^{(t-\tau)dA_0}}_{\mathcal{L}(Z_{2-},Z_{2})}\Big(\n{G_{2}^{\mu_0}(\bsym{z}^{\mu_0}(\tau))-G_{2}^{\mu_0}(\bsym{z}^0(\tau))}_{Z_{2-}}+\n{G_{2}^{\mu_0}(\bsym{z}^0(\tau))-G_{2}^0(\bsym{z}^0(\tau))}_{Z_{2-}}\Big)\Big\}d\tau\\
&+\int_0^t\Big\{\n{e^{-(t-\tau)Tr(m^{\mu_0})}}_{Z_{3+}}\n{G_3(\bsym{z}^{\mu_0}(\tau))-G_3(\bsym{z}^0(\tau))}_{Z_3}\\
&+\n{e^{-(t-\tau)Tr(m^{\mu_0})}-e^{-(t-\tau)m^0}}_{Z_3}\n{G_3(\bsym{z}^0(\tau))}_{Z_{3+}}\Big\}d\tau+\sum_{i=4}^5\int_0^t\n{G_i(\bsym{z}^{\mu_0}(\tau))-G_i(\bsym{z}^0(\tau))}_{Z_i}d\tau.
\end{align*}

Using Lemma \ref{H2E:sembasicest}, Lemma \ref{H2E:semestiron}, Lemma \ref{H2E:swallow}, Lemma \ref{H2E:lemiron} and Lemma \ref{H2E:ineq} we have
\begin{align*}
&f^{\mu_0}(t)\leq Ct^{2\theta}e^{t\lambda_{01,h}}\n{u_{01}}_{Z_{1+}}+C(1+t^{2\theta})\frac{1}{|\lambda_{01,h}^{\Omega}|^{\theta/2}}+Ct\frac{1}{|\lambda_{01,h}^{\Omega}|^{\theta/2}}\n{z_{03}}_{Z_{3+}}\\
&+CT^{2\theta}(1+N(T))\int_0^t\Big(1+\frac{1}{(t-\tau)^{3/4+2\theta}}\Big)\Big(\Big(1+\frac{1}{\tau^{2\theta}}\Big)f^{\mu_0}(\tau)+\frac{1}{|\lambda_{01,h}^{\Omega}|^{\theta/2}}\Big)d\tau\\
&+CT^{2\theta}(1+(N(T))^2)\int_0^t\Big(1+\frac{1}{(t-\tau)^{3/4+2\theta}}\Big)\Big(1+\frac{1}{\tau^{2\theta}}\Big)e^{(t-\tau)\lambda_{01,h}^{\Omega}}d\tau\\
&+C(1+N(T))\int_0^t\Big(1+\frac{1}{(t-\tau)^{1/2}}\Big)\Big(\Big(1+\frac{1}{\tau^{2\theta}}\Big)f^{\mu_0}(\tau)+\frac{1}{|\lambda_{01,h}^{\Omega}|^{\theta/2}}\Big)d\tau\\
&+C(1+N(T))\int_0^t\Big(1+\frac{1}{\tau^{2\theta}}\Big)f^{\mu_0}(\tau)+C(1+(N(T))^2)\frac{1}{|\lambda_{01,h}^{\Omega}|^{\theta/2}}\int_0^t(t-\tau)\Big(1+\frac{1}{\tau^{2\theta}}\Big)d\tau
\end{align*}
\begin{align*}
&\leq C(1+T)\Big(\frac{1}{|\lambda_{01,h}^{\Omega}|^{2\theta}}+\frac{1}{|\lambda_{01,h}^{\Omega}|^{\theta/2}}\Big)+C(1+T^{2\theta})(1+(N(T))^2)\Big\{\frac{1}{|\lambda_{01,h}^{\Omega}|^{\theta/2}}\int_0^t\Big(1+\frac{1}{(t-\tau)^{3/4+2\theta}}\\
&+\frac{1}{(t-\tau)^{1/2}}+(t-\tau)\Big(1+\frac{1}{\tau^{2\theta}}\Big)\Big)d\tau+\int_0^t\Big(1+\frac{1}{\tau^{3/4+2\theta}}\Big)\Big(1+\frac{1}{(t-\tau)^{2\theta}}\Big)e^{\tau\lambda_{01,h}^{\Omega}}d\tau\\
&+\int_0^t\Big(1+\frac{1}{(t-\tau)^{3/4+2\theta}}+\frac{1}{(t-\tau)^{1/2}}\Big)\Big(1+\frac{1}{\tau^{2\theta}}\Big)f^{\mu_0}(\tau)d\tau\Big\}\\
&\leq a(T)\Big(\frac{1}{|\lambda_{01,h}^{\Omega}|^{\theta/2}}+\frac{1}{|\lambda_{01,h}^{\Omega}|^{\frac{1/4-4\theta}{7/4+4\theta}}}\Big)+b(T)\int_0^t\frac{f^{\mu_0}(\tau)}{(t-\tau)^{3/4+2\theta}\tau^{2\theta}}d\tau.
\end{align*}

Using Lemma \eqref{H2E:Gronlem} (see \eqref{H2E:asssing2}) we get that
\begin{align*}
\n{f^{\mu_0}}_{L_{\infty}(0,T)}\leq a(T)\Big(\frac{1}{|\lambda_{01,h}^{\Omega}|^{\theta/2}}+\frac{1}{|\lambda_{01,h}^{\Omega}|^{\frac{1/4-4\theta}{7/4+4\theta}}}\Big)C\exp\Big(b(T)^{\frac{7/4+4\theta}{1/4-4\theta}}CT^{7/4+4\theta}\Big),
\end{align*}

from which we conclude that $\lim_{h\to0^+}\n{f^{\mu_0}}_{L_{\infty}(0,T)}=0$ since $|\lambda_{01,h}^{\Omega}|=(\pi/h)^2\to\infty$ as $h\to0$.

\section*{Acknowledgement}
The author would like to express his gratitude towards his PhD supervisors Philippe Lauren\c{c}ot and \\Dariusz Wrzosek for their constant encouragement and countless helpful remarks and towards his numerous colleagues for stimulating discussions. \\
The author was supported by the International Ph.D. Projects Programme of Foundation for Polish Science operated within the Innovative Economy Operational Programme 2007-2013 funded by European
Regional Development Fund (Ph.D. Programme: Mathematical Methods in Natural Sciences). \\
The article is supported by NCN grant no $2012/05/N/ST1/03115$.\\
This publication has been co-financed with the European Union funds by the European Social Fund.\\
Part of this research was carried out during the author's visit to the Institut de Math{\'e}matiques de Toulouse, Universit\'e Paul Sabatier, Toulouse III.

\end{document}